\documentclass[11pt,reqno]{amsart}
\usepackage{amsfonts}

\usepackage{graphics}
\usepackage{indentfirst}
\usepackage{cite}
\usepackage{latexsym}
\usepackage{amsmath,amsthm}
\usepackage{amssymb}
\usepackage[dvips]{epsfig}
\usepackage{amscd}
\usepackage{mathrsfs}

\usepackage{color}

\newtheorem{theorem}{Theorem}[section]
\newtheorem{remark}{Remark}[section]

\newtheorem{definition}{Definition}[section]
\newtheorem{lemma}[theorem]{Lemma}

\newtheorem{proposition}[theorem]{Proposition}

\newcommand{\n}{\rho}

\newcommand{\lm}{\lambda}

\renewcommand{\div}{ {\rm div }  }

\newcommand{\na}{\nabla }

\newcommand{\pa}{\partial}

\newcommand{\bt}{\begin{theorem}}
\newcommand{\bl}{\begin{lemma}}
\newcommand{\el}{\end{lemma}}
\newcommand{\et}{\end{theorem}}
\newcommand{\ga}{\gamma}

\newcommand{\OM}{\Omega}

\newcommand{\curl}{{\rm curl} }

\newcommand{\de}{\delta}

\newcommand{\la}{\label}
\newcommand{\si}{\sigma}

\newcommand{\ol}{\overline}

\newcommand{\bn}{\begin{eqnarray}}
\newcommand{\en}{\end{eqnarray}}
\newcommand{\bnn}{\begin{eqnarray*}}
\newcommand{\enn}{\end{eqnarray*}}

\newcommand{\bnnn}{\begin{eqnarray*}}
\newcommand{\ennn}{\end{eqnarray*}}
\newcommand{\ben}{\begin{enumerate}}
\newcommand{\een}{\end{enumerate}}
\newcommand{\T}{\mathbb{T}}

\newcommand{\ba}{\begin{aligned}}
\newcommand{\ea}{\end{aligned}}
\newcommand{\be}{\begin{equation}}
\newcommand{\ee}{\end{equation}}

\def\p{\partial}
\def\norm[#1]#2{\|#2\|_{#1}}

\def\lam{\lambda}
\def\ep{\varepsilon}

\def\r{\mathbb{R}}
\def\rr{\mathbb{R}^2}
\def\rrr{\mathbb{R}^3}

\makeatletter      
\@addtoreset{equation}{section}
\makeatother       

\allowdisplaybreaks[4]

\begin{document}

\date{}

\title[Global Strong Solutions to Compressible Navier-Stokes Equations]
{Global Strong Solutions to the Three-Dimensional Axisymmetric Compressible Navier-Stokes
Equations with Large Initial Data and Vacuum}

\author{Qinghao Lei}
\address{School of Mathematical Sciences, University of Chinese Academy of Sciences, Beijing 100049, P. R. China;
Institute of Applied Mathematics, Academy of Mathematics and Systems Science,
Chinese Academy of Sciences, Beijing 100190, P.R. China}
\email{leiqinghao22@mails.ucas.ac.cn}

\begin{abstract}
This paper investigates the three-dimensional axisymmetric compressible Navier-Stokes equations under slip boundary conditions in a cylindrical domain excluding the axis.
For initial density allowed to vanish, we establish the global existence and
large time asymptotic behavior of strong and weak solutions,
provided the shear viscosity is a positive constant and the bulk one is a power function of density with the power bigger than four-thirds.
It should be noted that these results are obtained without any restrictions on the size of initial data.
The key idea is to derive a pointwise estimate of the effective viscous flux by
exploiting the axisymmetry of the solutions, along with the conformal mapping and the pull back Green's function, and then to cancel out the singularity using the slip boundary conditions.
\end{abstract}



\maketitle

\section{Introduction and main results}
We consider the three-dimensional barotropic compressible Navier-Stokes equations:
\be\ba\la{ns}
\begin{cases}
\rho_t + \div(\rho \mathbf{u}) = 0,\\
(\n \mathbf{u})_t + \div(\n \mathbf{u}\otimes \mathbf{u}) -\mu \Delta \mathbf{u} 
-\na ( (\mu + \lm) \div \mathbf{u}) + \na P = 0,
\end{cases}
\ea\ee
where $t \ge 0$ is time, $x \in \OM \subset \rrr$ is the spatial coordinate.
The unknown functions $\n=\n(x,t)$ and $\mathbf{u}(x,t)=(u^1(x,t),u^2(x,t),u^3(x,t))$ represent the 
density and velocity of the fluid, respectively.
The pressure $P$ is given by
\be\ba\la{i1}
  P=a\n^\ga,
\ea\ee
with constants $a>0$ and $\ga > 1$.
The shear viscosity coefficient $\mu$ and bulk viscosity coefficient $\lam$ satisfy:
\be\la{i2}
0<\mu = \text{constant},\quad \lam(\n)=b \n^\beta,
\ee
where $b$ and $\beta$ are positive constants.
Without loss of generality, we assume that $a=b=1$.

The system is subject to the given initial data
\be\la{i30}
\n(x,0)=\n_0(x),\quad \n \mathbf{u}(x,0)= \mathbf{m}_0(x), \quad x\in \OM,
\ee
and slip boundary conditions:
\be\la{i3}\ba
\mathbf{u} \cdot n = 0,\ \ \curl \mathbf{u} \times n = -K \mathbf{u} \,\,\,\text{on}\,\,\, \partial\Omega,
\ea\ee
where $K=K(x)$ is a $3 \times 3$ symmetric matrix defined on $\p \OM$,
and $n=(n_1,n_2,n_3)$ denotes the unit outer normal vector to the boundary $\partial \Omega$.

There is a vast literature addressing the strong solvability of the
multidimensional compressible Navier-Stokes system with constant viscosity coefficients.
The local existence and uniqueness of classical solutions
were proved by Nash \cite{N} and Serrin \cite{S}, respectively, for strictly positive initial density.
The first result of global classical solutions was established by
Matsumura-Nishida \cite{MN1}, provided the initial data are close
to a non-vacuum equilibrium in the $H^s$-norm.
Later, Hoff \cite{H1,H3} studied the problem for discontinuous
initial data and developed new a priori estimates for the material derivative $\mathbf{\dot{u}}$.
For arbitrarily large initial data, Lions \cite{L2} (see also Feireisl \cite{F} and Feireisl et al. \cite{FNP})
proved the global existence of finite-energy weak solutions under the condition that the adiabatic exponent $\ga$ is suitably large.
Recently, Huang-Li-Xin \cite{HLX2} established the global
existence and uniqueness of classical solutions
to the three-dimensional Cauchy problem.
Their result holds for initial data with small total energy but possibly large oscillations and vacuum.
Subsequently, Li-Xin \cite{LX2} extended these existence results to the two-dimensional case and obtained the large time asymptotic behavior of solutions.
Furthermore, Cai-Li \cite{CL} generalized the above results to bounded domains 
with the velocity field subject to slip boundary conditions.

It is noteworthy that, without restrictions on the size of initial data,
a remarkable result was established by Vaigant-Kazhikhov \cite{VK}, who proved that the two-dimensional system (\ref{ns})--(\ref{i30}) admits a unique global strong solution for large initial data with density away from vacuum, provided $\beta>3$ in rectangle domains.
Later, in the periodic domain, Jiu-Wang-Xin \cite{JWX1} generalized the result in \cite{VK} by removing the condition that the initial density should be away from vacuum.
Recently, for the system (\ref{ns})--(\ref{i30}) in the two-dimensional periodic domains or the two-dimensional whole space with the density allowed to vanish,
Huang-Li \cite{HL2,HL3} (see also \cite{JWX2})
relaxed the crucial condition from $\beta>3$ to $\beta>\frac{4}{3}$
by applying some new ideas based on commutator theory and blow up criterion.
Very recently, Fan-Li-Li \cite{FLL} investigated the problem (\ref{ns})--(\ref{i30}) in a general
two-dimensional bounded simply connected domain, where the velocity field is subject to the Navier-slip boundary conditions.
They established the global existence of strong and weak solutions when $\beta>\frac{4}{3}$.
Furthermore, Fan-Li-Wang \cite{FLW} obtained the time-independent upper bound of the density and the exponential decay of the global strong solution under the sole assumption $\beta>\frac{4}{3}$ in two-dimensional periodic domains or bounded simply connected domains.
Later, Fan-Jiang-Li \cite{FJL} generalized these results to two-dimensional multi-connected domains.

In this paper, we investigate the global existence of axisymmetric
strong and weak solutions to the three-dimensional compressible Navier-Stokes equations in a cylindrical domain excluding the axis,
subject to slip boundary conditions.
Without loss of generality, we consider
\be\la{om}\ba
\OM=A \times \T,
\ea\ee
where $A=\{ (x_1,x_2) \in \rr : 1<x^2_1+x^2_2<4 \}$ is a two-dimensional annulus and $\T=\r / \mathbb{Z}$ is the one-dimensional torus.
We assume that the flow is periodic in the $x_3$-direction with period 1.

For $(x_1,x_2,x_3)\in \rrr$, we introduce the cylindrical coordinate transformation
\be\ba\nonumber
\begin{cases}
x_1=r \cos \theta,\\
x_2=r \sin \theta,\\
x_3=z,\\
\end{cases}
\ea\ee
and define the standard orthonormal basis in $\rrr$ as:
\be\ba\nonumber
\mathbf{e}_r=\frac{(x_1,x_2,0)}{r},
\ \mathbf{e}_\theta=\frac{(-x_2,x_1,0)}{r},
\  \mathbf{e}_z=(0,0,1).
\ea\ee
where $r=\sqrt{x^2_1+x^2_2}$.

A scalar function $g$ or a vector-valued function
$\mathbf{f}=f_r \mathbf{e}_r + f_\theta \mathbf{e}_\theta + f_z\mathbf{e}_z$
is called axisymmetric if $g$, $f_r$, $f_\theta$, and $f_z$ are independent of $\theta$.

We study the axisymmetric solutions to the problem $(\ref{ns})-(\ref{i3})$
that are periodic in $x_3$ with period $1$.
Specifically, we consider solutions of the form:
\be\la{sdc}\ba
\begin{cases}
\n(x_1,x_2,x_3,t)=\n(r,z,t), \\
\mathbf{u}(x_1,x_2,x_3,t)=u_r(r,z,t)\mathbf{e}_r + u_\theta(r,z,t)\mathbf{e}_\theta + u_z(r,z,t)\mathbf{e}_z, \\
\n(x_1,x_2,x_3+1,t) = \n(x_1,x_2,x_3,t), \ 
\mathbf{u}(x_1,x_2,x_3+1,t)=\mathbf{u}(x_1,x_2,x_3,t),
\end{cases}
\ea\ee
for any $(x_1,x_2)\in A$ and $x_3 \in \r$.

Before stating the main results, we first explain the notations
and conventions used throughout this paper. We denote
\be\ba\nonumber
  \int f dx = \int_{\OM} fdx,\quad
   \ol{f}=\frac{1}{|\OM|}\int fdx.
\ea\ee
For $1\leq r\leq\infty$, the standard Lebesgue and Sobolev spaces are denoted by
\be\ba\nonumber
\begin{cases}
L^r =L^r(\OM),\quad W^{s,r} =W^{s,r}(\OM),\quad H^s =W^{s,2}, \\
\tilde{H}^1 =\left\{v \in H^1(\Omega ) \big| v \cdot n=0, \mathrm{curl} v \times n =-K v \text{ on } \partial \Omega \right\}.
\end{cases}
\ea\ee

In the axisymmetric setting and through coordinate transformations,
we define the corresponding two-dimensional domain $D$ associated with the domain $\OM$ as:
\be\la{ewD}\ba
D=\{ (r,z)\in \rr:1<r<2,0<z<1 \}.
\ea\ee
The material derivative is given by
\be\ba\nonumber
\frac{D}{Dt}f=\dot{f} \triangleq f_t + \mathbf{u}\cdot\na f,
\ea\ee
and the shear stress tensor is defined as:
\be\ba\nonumber
D(v)=\frac{1}{2} \left( \na v + (\na v)^{ \text{tr} } \right).
\ea\ee

We now introduce the definitions of weak and strong solutions in the axisymmetric class for the system (\ref{ns}).
\begin{definition}
A pair $(\n,\mathbf{u})$ is called a weak solution in the axisymmetric class
to the system (\ref{ns}) if it is axisymmetric and periodic in $x_3$ with period $1$
(i.e., (\ref{sdc}) holds), and satisfies (\ref{ns}) in the sense of distribution.

Furthermore, such a weak solution in the axisymmetric class is called a strong solution in the axisymmetric class
if all derivatives involved in (\ref{ns}) are regular distributions,
and the system (\ref{ns}) holds almost everywhere in
$\OM\times(0,T)$.
\end{definition}
The first main result concerning the global existence and exponential decay of strong solutions can be described as follows:
\begin{theorem}\la{th0}
Assume that
\be\la{bg}\ba
\beta>\frac{4}{3}, \quad \ga>1,
\ea\ee
and that $K$ is a smooth, symmetric, positive semi-definite $3 \times 3$ axisymmetric matrix-valued function
satisfying $K + 2D(n)$ is positive definite on some subset $\Sigma \subset \p \OM$ with $|\Sigma|>0$.
Suppose that the initial data $(\n_0,\mathbf{m}_0)$ satisfy for some $q>3$,
\be\la{ssol1}\ba
0\le \n_0 \in W^{1,q},\quad \mathbf{u}_0 \in \tilde{H}^1,
\quad \mathbf{m}_0(x) = \n_0 \mathbf{u}_0,
\ea\ee
and $\n_0$, $\mathbf{u}_0$ are axisymmetric and periodic in $x_3$ with period $1$.

Then the problem $(\ref{ns})-(\ref{i3})$ admits a unique
strong solution $(\n,\mathbf{u})$ within the axisymmetric class
in $\OM \times (0,\infty)$ satisfying for any $0<T<\infty$,
\be\la{ssol2}\ba
\begin{cases}
\rho\in C([0,T];W^{1,q} ), \quad \n_t\in L^\infty(0,T;L^2), \\ 
\mathbf{u} \in L^\infty(0,T; H^1) \cap L^{(q+1)/q}(0,T; W^{2,q}), \\ 
t^{1/2} \mathbf{u} \in L^2(0,T; W^{2,q}) \cap L^\infty(0,T;H^2), \\
t^{1/2} \mathbf{u}_t \in L^2(0,T;H^1), \\
\n \mathbf{u} \in C([0,T];L^2), \quad \sqrt{\n} \mathbf{u}_t\in L^2(\OM \times(0,T)).
\end{cases} 
\ea\ee

Moreover, the global solution $(\n,\mathbf{u})$ satisfies the following properties:

1) (Uniform boundedness) There exists a positive constant $C$ depending only on
$\ga$, $\beta$, $\mu$, $\| \n_0 \|_{L^\infty}$, $\| \mathbf{u}_0 \|_{H^1}$, and $K$, such that for any $0<T<\infty$,
\be\la{up}\ba
\sup_{0 \le t \le T} \| \n(\cdot,t) \|_{L^\infty} \le C.
\ea\ee

2) (Exponential decay) For any $p \in [1,\infty)$, there exist positive constants
$C$ and $\alpha_0$ depending only on
$p$, $\ga$, $\beta$, $\mu$, $\| \n_0 \|_{L^\infty}$, $\| \mathbf{u}_0 \|_{H^1}$, and $K$, such that for any $1 \le t <\infty$,
\be\la{ed}\ba
\| \n(\cdot,t)-\ol{\n_0}\|_{L^p} +
\| \na \mathbf{u}(\cdot,t) \|_{L^p} \le C e^{-\alpha_0 t}.
\ea\ee
\end{theorem}

The second result establishes the global existence and exponential decay of weak solutions.

\begin{theorem}\la{th1}
Suppose that the conditions of Theorem \ref{th0} hold, with $\n_0 \in W^{1,q}$ in (\ref{ssol1}) replaced by $\n_0 \in L^\infty$.
Then, there exists at least one weak solution $(\n,\mathbf{u})$ of the problem $(\ref{ns})-(\ref{i3})$
within the axisymmetric class in $\OM \times (0,\infty)$ satisfying,
for any $0<T<\infty$ and $1 \le p < \infty$,
\be\la{wsol2}\ba
\begin{cases}
\rho\in L^{\infty}(\OM \times (0,\infty)) \cap C([0,\infty);L^p), \\ 
\mathbf{u} \in L^2(0,\infty;H^1) \cap L^\infty(0,\infty;H^1), \\
t^{1/2}\mathbf{u}_t \in L^2(0,T;L^2), t^{1/2} \na \mathbf{u} \in L^\infty(0,T;L^p).
\end{cases}
\ea\ee

Furthermore, the weak solution $(\n,\mathbf{u})$ satisfies the estimates (\ref{up}) and (\ref{ed}).
\end{theorem}

Finally, similar to \cite{CL,LX}, we can deduce from $(\ref{ed})$
the following large-time behavior of the spatial gradient of the density for the strong solution in Theorem $\ref{th0}$ when vacuum states appear initially.   
\begin{theorem}\la{th3}
In addition to the assumptions in Theorem \ref{th0}, 
we further assume that there exists some point $x_0\in \OM$ such that $\n_0(x_0)=0$. 
Then for any $r>2$, there exists a positive constant $C$ depending only on
$r$, $\ga$, $\beta$, $\mu$, $\| \mathbf{u}_0 \|_{H^1}$, $\| \rho_0 \|_{L^1\cap L^\infty}$,
and $K$, such that for any $t \ge 1$
\be\la{pbu0}\ba
\| \na \n(\cdot ,t) \|_{L^r} \ge C e^{\alpha_0 \frac{r-2}{r} t }.
\ea\ee
\end{theorem}

A few remarks are in order.

\begin{remark}\la{lrk1}
For bounded domains, the usual Navier-type slip condition can be stated as follows:
\be\la{NSB}\ba
\mathbf{u} \cdot n = 0, \quad ( 2D(\mathbf{u})n + \vartheta \mathbf{u} )_{ \textnormal{tan} } =0 \  \textnormal{ on } \p \OM,
\ea\ee
where $\vartheta$ is a scalar friction function that measures the tendency of the fluid to slip on the boundary,
and the symbol $\mathbf{w}_\textup{tan}$ represents the projection of tangent plane of the vector $\mathbf{w}$ on $\p \OM$.
As shown in \cite[Remark 1.1]{CL}, the Navier-type slip condition (\ref{NSB})
is a special case of the slip boundary condition (\ref{i3}).
\end{remark}

\begin{remark}\la{lrk2}
Under the assumptions of Theorem \ref{th0}, if the initial data $(\n_0,\mathbf{m}_0)$ further satisfy for some $q>2$,
\be\la{csol1}\ba
0\le \n_0 \in W^{2,q},\quad \mathbf{u}_0 \in H^2 \cap \tilde{H}^1,
\quad \mathbf{m}_0(x) = \n_0 \mathbf{u}_0,
\ea\ee
and the compatibility condition
\be\la{csol2}\ba
- \mu \Delta \mathbf{u}_0 - \nabla( (\mu + \lm(\n_0) ) \div \mathbf{u}_0) +  \nabla P(\n_0)=\n_0^{1/2}g, \ \textnormal{ for some } g \in L^2,
\ea\ee
then the strong solution obtained in Theorem \ref{th0} becomes a classical one for positive time.
The detailed proofs follow arguments similar to those in \cite{JWX1,LZZ,HL,HLX2}.
\end{remark}

\begin{remark}\la{lrk3}
Theorems \ref{th0} and \ref{th1} improve the results of Wang-Li-Guo \cite{WLG},
who studied the problem (\ref{ns})--(\ref{i30}) in a periodic domain excluding the axis.
Under the assumptions that $\beta>2$ and the initial density is strictly positive,
they proved that the system (\ref{ns})--(\ref{i30}) admits a unique
global axisymmetric classical solution $(\n,\mathbf{u})$ with $u_\theta=0$.
\end{remark}

\begin{remark}\la{lrk4}
It is worth noting that under the assumption of axisymmetry and the condition that the domain $\OM$ excludes the axis,
our problem effectively reduces to a two-dimensional case.
Hence, our results are consistent with those for the two-dimensional case in \cite{FLW,FLL,FMN,HL2}.
\end{remark}

We now make some comments on the analysis of this paper. 
For smooth initial data away from vacuum, the local well-posedness of strong solutions to the problem (\ref{ns})--(\ref{i3})
was established in \cite{SS,SVA}.
To extend the strong solution globally in time while allowing vacuum,
we need to derive global a priori estimates for smooth solutions to (\ref{ns})--(\ref{i3})
in suitable higher order norms, independent of the lower bound of the initial density.
Motivated by \cite{HL2,FLL,FLW}, we find that the key issue is to obtain the uniform upper bound of the density.
First, by combining the two-dimensional Gagliardo-Nirenberg inequality
with axisymmetry and the fact that the domain excludes the axis,
we establish Gagliardo-Nirenberg-Sobolev inequalities in the three-dimensional axisymmetric domain $\OM$ analogous to the two-dimensional case.
These inequalities play a crucial role in subsequent estimates.
On the other hand, since the domain excludes the axis, it is multi-connected.
As shown in \cite{WWV}, the usual div-curl type estimate
\be\ba\nonumber
\| \na v \|_{L^2} \le C \left( \| \div v \|_{L^2} + \| \curl v \|_{L^2} \right)
\quad \text{ for } v \in H^1 \text{ with } (v \cdot n)|_{\p \OM}=0,
\ea\ee
no longer holds. This poses an obstacle to our analysis.
To overcome this difficulty, based on \cite{AACG,CL}, we establish the following estimate (see Lemmas \ref{dltdc1} and \ref{dltdc2}):
\be\la{dddc}\ba
\| \na v \|^2_{L^2} \le C \left( 2 \| \div v \|^2_{L^2} + \| \curl v \|^2_{L^2} + \int_{\p \OM} v \cdot K \cdot v ds \right),
\ea\ee
provided $v \in H^1$ with $(v \cdot n)|_{\p \OM}=0$ and
$K$ satisfies the assumptions in Theorem \ref{th0}.
By virtue of (\ref{dddc}), we first derive the standard energy estimate (\ref{kv01}).
Then, combining this with Lemma \ref{ewgn} and following a procedure
analogous to the proof in \cite{FLW},
we obtain the time-uniform estimates (\ref{yzgj1}) and (\ref{yzgj2}).
These estimates are essential to derive the time-uniform bound of the density.

As in \cite{HL2,FLL,FLW}, the key to obtaining the upper bound of the density is to estimate the $L^\infty$-norm of the effective viscous flux $G$ (see (\ref{gw}) for its definition).
In view of the slip boundary conditions and $(\ref{ns})_2$, we deduce that $G$ satisfies the elliptic equation (\ref{evf1}).
Using axisymmetry, we transform equation (\ref{evf1}) into its two-dimensional form (\ref{evf2}).
Subsequently, with the help of Green's function for the two-dimensional unit disk and a conformal mapping, we derive the pointwise estimate of $G$ (see Lemmas \ref{l7} and \ref{ll7}).
Through a series of careful calculations, we obtain the desired time-uniform upper bound of the density, provided $\beta>\frac{4}{3}$; see Lemma \ref{l9} and its proof.

Furthermore, in deriving the preceding estimates, the treatment of
boundary terms relies crucially on two key observations from \cite{CL}:
\be\la{bjds}\ba
\mathbf{u} = -(\mathbf{u} \times n) \times n \triangleq \mathbf{u}^\bot \times n,
\quad (\mathbf{u} \cdot \na) \mathbf{u} \cdot n=-(\mathbf{u} \cdot \na) n \cdot \mathbf{u} \quad \text{ on } \p \OM,
\ea\ee
which hold under the condition that $\mathbf{u} \cdot n = 0$ on $\p \OM$.
Finally, using the upper bound of the density established above and following arguments similar to those in \cite{FLL,HL2,LLL},
we derive the exponential decay and higher-order derivative estimates of the solution, allowing us to extend the local solution globally.

The rest of this paper is structured as follows:
Section 2 introduces some known facts and essential inequalities for the subsequent analysis.
Section 3 focuses on deriving the time-uniform upper bound of the density.
Section 4 establishes higher-order derivative estimates based on the previously obtained density bound.
Finally, Section 5 presents the proofs of the main results, Theorems \ref{th0}--\ref{th3}.

\section{Preliminaries}
This section collects essential known facts and inequalities that will be used throughout this paper.

We begin with the following local existence result of the strong solution;
its proof can be found in \cite{SS,SVA}.
\begin{lemma}\la{lct}
Assume that the initial data $\left( \n_0,\mathbf{m}_0 \right)$ satisfies
\be\la{lct01}\ba
\n_0 \in H^2, \  \inf\limits_{x\in\OM}\n_0(x) >0, \  \mathbf{u}_0 \in H^2 \cap \tilde{H}^1, \ \mathbf{m}_0 = \n_0 \mathbf{u}_0.
\ea\ee
Then there is a small time $T>0$ and a constant $C_0>0$ both depending only on
$\mu$, $\ga$, $\beta$, $K$, $\| \n_0\|_{H^2}$, $\| \mathbf{u}_0 \|_{H^2}$,
and $\inf\limits_{x\in\OM}\n_0(x)$,
such that the problem $(\ref{ns})-(\ref{i3})$ admits a unique strong solution $(\n,\mathbf{u})$ in $\OM \times (0,T]$ satisfying
\be\la{lct1}\ba
\begin{cases}
\rho\in C([0,T];H^2), \quad \n_t\in C([0,T];H^1), \\
\mathbf{u} \in L^2(0,T; H^3), \quad \mathbf{u}_t \in L^2(0,T;H^2) \cap H^1(0,T;L^2),
\end{cases} 
\ea\ee
and
\be\la{lct2}\ba
\inf\limits_{(x,t)\in \OM \times (0,T)}\n(x,t) \ge C_0 >0.
\ea\ee
\end{lemma}

Based on \cite[Lemma 2]{GWX} and the rotation and transformation invariance of (\ref{ns})--(\ref{i2}), we obtain the following result:
\begin{lemma}\la{dcx}
Assume that the initial data is axisymmetric and periodic in $x_3$ with period $1$.
Then the local strong solution of $(\ref{ns})-(\ref{i3})$ is also axisymmetric and periodic in $x_3$ with period $1$.
\end{lemma}

We recall the following Gagliardo-Nirenberg inequality from \cite{TG}.
\begin{lemma}\la{gn1}
Let $D$ be a bounded Lipschitz domain in $\rr$.
For $p\in[2,\infty)$, there exists a positive constant C depending only on $D$ such that for any $v\in H^1(D)$,
\be\ba\la{gn11}
\| v \|_{L^p(D)} \le C p^{1/2}\| v \|^{2/p}_{L^2(D)} \| v \|^{1-2/p}_{H^1(D)}.
\ea\ee
\end{lemma}

For three-dimensional axisymmetric functions, we establish the following Gagliardo-Nirenberg-Sobolev inequalities, which play a crucial role in our subsequent analysis.
\begin{lemma}\la{ewgn}
Let $\OM$ be as in $(\ref{om})$,
and let $\mathbf{f}$ and $g$ be vector-valued and scalar axisymmetric functions on $\OM$, respectively.
Then, for any $p\in [2,\infty)$, $q\in[1,2)$, and $r\in(2,\infty)$,
there exists a generic constant $C>0$ that may depend on $q$ and $r$ such that
\be\la{ewgn01}\ba
\| \mathbf{f} \|_{L^p}
\le C p^{1/2} \| \mathbf{f} \|^{\frac{2}{p}}_{L^2}
\| \mathbf{f} \|^{1-\frac{2}{p}}_{H^1},\quad
\| \mathbf{f} \|_{L^{\frac{2q}{2-q}}}
\le C \| \mathbf{f} \|_{W^{1,q}},\quad
\| \mathbf{f} \|_{L^{\infty}}
\le C \| \mathbf{f} \|_{W^{1,r}},
\ea\ee
\be\la{ewgn02}\ba
\| g \|_{L^p}
\le C p^{1/2} \| g \|^{\frac{2}{p}}_{L^2} \| g \|^{1-\frac{2}{p}}_{H^1},\quad
\| g \|_{L^{\frac{2q}{2-q}}} \le C \| g \|_{W^{1,q}},\quad
\| g \|_{L^{\infty}} \le C \| g \|_{W^{1,r}}.
\ea\ee
\end{lemma}
\begin{proof}
First, for the axisymmetric vector-valued function $\mathbf{f}$ on $\OM$,
it can be expressed in the standard orthonormal basis as:
\be\ba\la{ewgn100}
\mathbf{f}(x)=f_r(r,z)\mathbf{e}_r + f_\theta(r,z)\mathbf{e}_\theta + f_z(r,z)\mathbf{e}_z,
\ea\ee
which implies that for any $2 \le p <\infty$,
\be\la{ewgn1}\ba
\int_\OM |\mathbf{f}|^p dx = 2 \pi \int_{D} r |\mathbf{f}|^p dr dz
& \le C \int_{D} r \left( |f_r|^p + |f_\theta|^p + |f_z|^p \right) dr dz.
\ea\ee
Moreover, we deduce from (\ref{gn11}) that
\be\la{ewgn2}\ba
\int_{D} r |f_r|^p dr dz
& \le C p^{p/2} \| r^{\frac{1}{p}} f_r \|^2_{L^2(D)} \| r^{\frac{1}{p}} f_r \|^{p-2}_{H^1(D)} \\
& \le C  p^{p/2} \| r^{\frac{1}{p}} f_r \|^p_{L^2(D)}
+ C  p^{p/2} \| r^{\frac{1}{p}} f_r \|^2_{L^2(D)} \| \tilde{\na} (r^{\frac{1}{p}} f_r) \|^{p-2}_{L^2(D)},
\ea\ee
where $\tilde{\na} \triangleq (\p_r,\p_z)$.

A direct computation from (\ref{ewgn100}) gives
\be\la{ewgn3}\ba
| \na \mathbf{f} |^2 = (\p_r f_r)^2 + (\p_z f_r)^2 + (\p_r f_\theta)^2 + (\p_z f_\theta)^2
+ (\p_r f_z)^2 + (\p_z f_z)^2 + \frac{f^2_r + f^2_\theta}{r^2}.
\ea\ee
Combining this with (\ref{ewgn2}) yields
\be\ba\nonumber
\int_{D} r |f_r|^p dr dz
& \le C  p^{p/2} \| r^{\frac{1}{p}} f_r \|^p_{L^2(D)}
+ C  p^{p/2} \| r^{\frac{1}{p}} f_r \|^2_{L^2(D)} \| \tilde{\na} (r^{\frac{1}{p}} f_r) \|^{p-2}_{L^2(D)} \\
& \le C  p^{p/2} \| \mathbf{f} \|^p_{L^2(\OM)} + C  p^{p/2} \| \mathbf{f} \|^2_{L^2(\OM)}
\| \na \mathbf{f} \|^{p-2}_{L^2(\OM)} \\
& \le C  p^{p/2} \| \mathbf{f} \|^2_{L^2(\OM)} \| \mathbf{f} \|^{p-2}_{H^1(\OM)}.
\ea\ee
Similarly, we also have
\be\ba\nonumber
\int_{D} r \left( |f_\theta|^p + |f_z|^p \right) dr dz
\le C  p^{p/2} \| \mathbf{f} \|^2_{L^2(\OM)} \| \mathbf{f} \|^{p-2}_{H^1(\OM)},
\ea\ee
which together with (\ref{ewgn1}) implies
\be\la{ewgn4}\ba
\| \mathbf{f} \|_{L^p(\OM)} 
\le C  p^{1/2} \| \mathbf{f} \|^{\frac{2}{p}}_{L^2(\OM)}
\| \mathbf{f} \|^{1-\frac{2}{p}}_{H^1(\OM)}.
\ea\ee
In addition, by virtue of (\ref{ewgn3}) and Sobolev inequality, we can obtain
\be\la{ewgn5}\ba
\| \mathbf{f} \|_{L^{\frac{2q}{2-q}}(\OM)}
\le C \| \mathbf{f} \|_{W^{1,q}(\OM)},\quad
\| \mathbf{f} \|_{L^{\infty}(\OM)}
\le C \| \mathbf{f} \|_{W^{1,r}(\OM)}.
\ea\ee

On the other hand, for the axisymmetric scalar function $g$ on $\OM$ satisfying $g(x_1,x_2,x_3)=g(r,z)$,
a direct calculation yields:
\be\ba\nonumber
\na g = \p_r g \mathbf{e}_r + \p_z g \mathbf{e}_z,
\ea\ee
which shows that
\be\la{ewgn6}\ba
|\na g|^2=|\p_r g|^2 + |\p_z g|^2.
\ea\ee
Similar to (\ref{ewgn2}), we derive
\be\ba\nonumber
\int_\OM |g|^p dx = 2\pi \int_{D} r |g|^p dr dz
&\le C p^{p/2} \| r^{\frac{1}{p}} g \|^2_{L^2(D)} \| r^{\frac{1}{p}} g \|^{p-2}_{H^1(D)} \\
& \le C p^{p/2} \| r^{\frac{1}{p}} g \|^p_{L^2(D)}
+ C p^{p/2} \| r^{\frac{1}{p}} g \|^2_{L^2(D)} \| \tilde{\na} (r^{\frac{1}{p}} g) \|^{p-2}_{L^2(D)} \\
& \le C p^{p/2} \| g \|^p_{L^2(\OM)} + C p^{p/2} \| g \|^{2}_{L^2(\OM)} \| \na g \|^{p-2}_{L^2(\OM)},
\ea\ee
which implies
\be\la{ewgn7}\ba
\| g \|_{L^p(\OM)} 
\le C p^{1/2} \| g \|^{\frac{2}{p}}_{L^2(\OM)} \| g \|^{1-\frac{2}{p}}_{H^1(\OM)}.
\ea\ee
Furthermore, by (\ref{ewgn6}) and the two-dimensional Sobolev inequality, we have
\be\la{ewgn8}\ba
\| g \|_{L^{\frac{2q}{2-q}}(\OM)}
\le C \| g \|_{W^{1,q}(\OM)},\quad
\| g \|_{L^{\infty}(\OM)}
\le C \| g \|_{W^{1,r}(\OM)}.
\ea\ee
The combination of (\ref{ewgn4}), (\ref{ewgn5}), (\ref{ewgn7}) and (\ref{ewgn8}) yields
(\ref{ewgn01}) and (\ref{ewgn02}) and completes the proof of Lemma \ref{ewgn}.
\end{proof}

The following Poincar\'e type inequality can be found in \cite{F}.
\begin{lemma}\la{pt}
Let $v\in H^1$, and let $\n$ be a non-negative function satisfying
\be\ba\nonumber
0<M_1\leq \int \n dx,\quad \int \n^r dx \leq M_2,
\ea\ee
with $r>1$. Then there exists a positive constant $C$ depending only on 
$M_1$, $M_2$, and $r$ such that
\be\ba\la{pt1}
\|v\|_{L^2}^2 \leq C\int \n |v|^2 dx + C \|\na v\|_{L^2}^2.
\ea\ee
\end{lemma}

The following div-curl estimate will be frequently used in later
arguments and can be found in \cite{AJ,WWV}.
\begin{lemma}\la{dc}
Let $k \ge 0$ be an integer and $1<q<\infty$.
Assume that $\OM$ is a bounded domain in $\rrr$ and its $C^{k+1,1}$ boundary $\p \OM$ only has a finity number of 2-dimensional connected components.
Then, for $v \in W^{k+1,q}(\OM)$ with $(v \cdot n)|_{\p \OM} = 0$ or $(v \times n)|_{\p \OM} = 0$,
there exists a positive constant $C$ depending only on $k$, $q$ and $\OM$ such that
\be\ba\la{dc1}
\| v \|_{W^{k+1,q}(\OM)} \le C\left( \|\div v\|_{W^{k,q}(\OM)} + \| \curl v \|_{W^{k,q}(\OM)} + \| v \|_{L^q(\OM)} \right).
\ea\ee
\end{lemma}

The following Lemmas \ref{dltdc1} and \ref{dltdc2} are essential for deriving the uniform estimates.
\begin{lemma}\la{dltdc1}
Let $\OM$ be an axisymmetric and bounded Lipschitz domain in $\rrr$.
Then for $v \in H^1$ with $v \cdot n=0$ on $\p \OM$ and
smooth positive semi-definite $3 \times 3$ symmetric matrix $B$ satisfying
$B>0$ on some $\Sigma \subset \p \OM$ with $|\Sigma|>0$,
there exists a positive constant $\Lambda$ depending only on $\OM$, such that
\be\la{dltdc01}\ba
\| v \|^2_{H^1} \le \Lambda \left( \| D(v) \|^2_{L^2} + \int_{\p \OM} v \cdot B \cdot v ds \right).
\ea\ee
\end{lemma}
\begin{proof}
We prove (\ref{dltdc01}) by contradiction.
If (\ref{dltdc01}) fails, then there exists a sequence
$\{ v_m \}_{m \in \mathbb{N}} \subset H^1$ with $v_m \cdot n=0$ on $\p \OM$ such that
\be\la{ddc1}\ba
\| v_m \|^2_{H^1} > m \left( \| D(v_m) \|^2_{L^2} + \int_{\p \OM} v_m \cdot B \cdot v_m ds \right).
\ea\ee

Normalize the sequence by setting $||| v_m |||=1$, where $||| v_m ||| \triangleq \| v_m \|_{L^2}+\| D(v_m) \|_{L^2}$.
From Korn's inequality (see \cite{NJA}), we have
\be\la{ddc2}\ba
\| v_m \|_{H^1} \le C \left( \| v_m \|_{L^2} + \| D(v_m) \|_{L^2} \right) \le C,
\ea\ee
so $\{ v_m \}_{m \in \mathbb{N}}$ is bounded in $H^1$.
The Sobolev compact embedding theorem then yields a subsequence $\{ v_{m_i} \}_{i \in \mathbb{N}}$ and $v \in H^1$ with $v \cdot n=0$ on $\p \OM$ such that
\be\la{ddc3}\ba
v_{m_i} \rightharpoonup v  \ \text{ in } H^1(\OM) \cap H^{\frac{1}{2}}(\p \OM),
\quad v_{m_i} \to v  \ \text{ in } L^2(\OM) \cap L^2(\p \OM).
\ea\ee
Combining this with (\ref{ddc1}) and (\ref{ddc2}) yields $D(v)=0$ in $\OM$.
By \cite[Proposition 3.13]{AACG}, we conclude that there exist constant vectors $\mathbf{b}$ and $\mathbf{c}$ such that $v=\mathbf{b} \times \mathbf{x} + \mathbf{c}$.
The boundary condition $v \cdot n=0$ on $\p \OM$ implies $\mathbf{c}=0$.

Moreover, (\ref{ddc1}), (\ref{ddc2}) and (\ref{ddc3}) ensure that
\be\ba\nonumber
\int_{\Sigma} v \cdot B \cdot v ds =0.
\ea\ee
Since $B>0$ on $\Sigma$, it follows that $v=\mathbf{b} \times \mathbf{x}=0$ on $\Sigma$,
which implies $\mathbf{b}=0$, hence $v=0$ in $\OM$.

However, (\ref{ddc1}) and (\ref{ddc3}) imply that $||| v |||=1$, leading to a contradiction.
Thus, (\ref{dltdc01}) holds, and the proof is finished.
\end{proof}

The following lemma can be found in \cite[Lemma 6.2]{CL}.
\begin{lemma}\la{dltdc2}
Let $\OM$ be a smooth bounded domain in $\rrr$. Then for $v \in H^2(\OM)$ with $v \cdot n=0$ on $\p \OM$, it holds that
\be\la{dltdc02}\ba
2\int D(v) \cdot D(v) dx = 2\int (\div v)^2 dx + \int |\curl v|^2 dx -2\int_{\p \OM} v \cdot D(n) \cdot v ds.
\ea\ee
\end{lemma}

Combining Lemma \ref{dltdc1} and Lemma \ref{dltdc2}, we obtain the following weighted div-curl type estimate.
\begin{lemma}\la{wdc}
Let $\OM$ be as in $(\ref{om})$, and let $K$ satisfy the assumptions in Theorem \ref{th0}.
Then for any $v \in H^2(\OM)$ with $(v \cdot n)|_{\p \OM} = 0$,
there exist positive constants $C$ and $\hat{\nu}$,
depending only on $\OM$, such that for any $\nu \in (0,\hat{\nu})$,
\be\ba\la{wdc1}
\int_\OM |v|^{\nu} |\na v|^2 dx
\le C \int_\OM |v|^{\nu} \left( (\div v)^2 + |\curl v|^2 \right) dx
+ C \int_{\p \OM} v \cdot K \cdot v |v|^\nu ds.
\ea\ee
\end{lemma}
\begin{proof}
First, using Cauchy's inequality, we directly calculate that
\be\la{wdc01}\ba
\left( \div ( |v|^{\frac{\nu}{2}} v ) \right)^2 \le 2 |v|^{\nu} (\div v)^2
+ \nu^2 |v|^{\nu} |\na v|^2, \\
\left| \curl ( |v|^{\frac{\nu}{2}} v ) \right|^2 \le 2 |v|^{\nu} |\curl v|^2
+ \nu^2 |v|^{\nu} |\na v|^2, \\
\left| \na ( |v|^{\frac{\nu}{2}} v ) \right|^2 \ge \frac{1}{2} |v|^{\nu} |\na v|^2
- \nu^2 |v|^{\nu} |\na v|^2.
\ea\ee

Observing that $|v|^{\frac{\nu}{2}} v \cdot n=0$ on $\p \OM$,
we select $B=K+2D(n)$ in Lemma \ref{dltdc1} and apply Lemma \ref{dltdc2} to derive
\be\la{wdc02}\ba
\int_\OM \left| \na ( |v|^{\frac{\nu}{2}} v ) \right|^2 dx
& \le C \int_\OM \left( \left( \div ( |v|^{\frac{\nu}{2}} v ) \right)^2
+ \left| \curl ( |v|^{\frac{\nu}{2}} v ) \right|^2 \right) dx \\
& \quad + C \int_{\p \OM} v \cdot K \cdot v |v|^\nu ds.
\ea\ee
Combining (\ref{wdc02}) with (\ref{wdc01}) and choosing $\hat{\nu}>0$ sufficiently small yields (\ref{wdc1}) for any $\nu \in (0,\hat{\nu})$,
which completes the proof.
\end{proof}

To estimate $\| \na \mathbf{u}\|_{L^{\infty}}$ and $\| \na \n\|_{L^{q}}$,
we require the following Beale-Kato-Majda type inequality, 
which was established in \cite{K} when $\div \mathbf{u} \equiv 0$. 
We refer readers to \cite{BKM,CL} for further details.
\begin{lemma}\la{bkm}
Let $\OM$ be a bounded domain in $\rrr$ with smooth boundary.
For $3<q<\infty$, there exists a positive constant $C$ depending only on $q$ and $\OM$ such that the following estimate holds:
\be\ba\nonumber
\|\na \mathbf{u}\|_{L^\infty} \le C \left( \|\div \mathbf{u} \|_{L^\infty}
+ \|\curl \mathbf{u}\|_{L^\infty} \right) \log \left(e+ \|\na^2 \mathbf{u}\|_{L^q} \right)+ C\|\na \mathbf{u}\|_{L^2}+C,
\ea\ee
for any function $\mathbf{u} \in \left\{ W^{2, q}(\Omega ) \big| \mathbf{u} \cdot n=0, \mathrm {curl}\,\mathbf{u} \times n=-K\mathbf{u} \ \textnormal{ on } \partial \Omega \right\}$.
\end{lemma}

Next, we introduce the following "inversion" operator of divergence; the proof can be found in \cite{CL}.
\begin{lemma}\la{iod}
Let $1<p<\infty$. There exists a bounded linear operator $\mathcal{B}$,
\be\ba\nonumber
\mathcal{B}:\left\{f \in L^p(\OM) : \  \int_\Omega fdx=0\right\} & \rightarrow ( W^{1,p}_0(\OM) )^3,
\ea\ee
such that $v=\mathcal{B}(f)$ satisfies
\be\la{iod1}\ba
\begin{cases}
\mathrm{div}v=f&\ \textnormal{ in } \Omega,\\
v=0&\ \textnormal{ on }\partial\Omega.
\end{cases}
\ea\ee
Moreover, the operator possesses the following properties:

(1) For $1<p<\infty$, there is a constant $C(p)$ depending only on $\Omega$ and $p$, such that
\bnn
\|\mathcal{B}(f)\|_{W^{1,p}}\leq C(p) \|f\|_{L^p}.
\enn

(2) If $f=\mathrm{div}h$, for some $h\in L^p$ with $h\cdot n=0$ on $\partial\Omega$, then $v=\mathcal{B}(f)$ is a weak solution of (\ref{iod1}) and satisfies
\bnn
\|\mathcal{B}(f)\|_{L^p}\leq C(p) \| h \|_{L^p}.
\enn
\end{lemma}

Finally, we state the following Zlotnik inequality, which plays an important role in deriving the uniform upper bound of the density; see \cite{ZAA}.
\begin{lemma}\la{zli}
Suppose that the function $y(t) \in W^{1,1}(0,T)$ satisfies
\bnn
y'(t)= g(y)+h'(t) \mbox{ on  } [0,T] ,\quad y(0)=y_0, 
\enn
with $ g\in C(\r)$ and $h\in W^{1,1}(0,T).$ If $g(\infty)=-\infty$
and 
\bnn 
h(t_2)-h(t_1) \le N_0 +N_1(t_2-t_1),
\enn
for all $0 \le t_1<t_2\le T$
with some $N_0\ge 0$ and $N_1\ge 0$, then
\bnn
y(t)\le \max\left\{y_0,\overline{\zeta} \right\}+N_0<\infty
\textnormal{ on } [0,T],
\enn
where $\overline{\zeta}$ is a constant such
that 
\bnn
g(\zeta)\le -N_1 \ \textnormal{ for } \ \zeta\ge \overline{\zeta}.
\enn
\end{lemma}

\section{A Priori Estimates (\uppercase\expandafter{\romannumeral1}): Upper Bound of the density}
In this section, we assume that $(\n,\mathbf{u})$ is the axisymmetric strong solution of (\ref{ns})--(\ref{i3}) on $\OM \times (0,T]$,
satisfying (\ref{sdc}), (\ref{lct1}) and (\ref{lct2}), whose existence is guaranteed by Lemmas \ref{lct} and \ref{dcx}.

We set
\be\ba\nonumber
A_1^2(t) \triangleq \int (2\mu + \lam(\n)) (\div \mathbf{u})^2 + |\na \mathbf{u}|^2 + (\n+1)^{\ga-1} (\n-\ol{\n})^2 dx,
\ea\ee
\be\ba\nonumber
A_2^2(t) \triangleq \int \rho(t)|\dot{\mathbf{u}}(t)|^2dx,
\ea\ee
and
\be\ba\nonumber
R_T \triangleq 1 + \sup_{0 \le t \le T} \| \n(t) \|_{L^\infty}.
\ea\ee

We begin with the standard energy estimate.
\begin{lemma}\la{l1}
There exists a positive constant
$C$ depending only  on $\mu$, $\gamma$, $\| \n_0 \|_{L^\infty}$, $\| \mathbf{u}_0 \|_{H^1}$, and $K$ such that
\be\ba\la{kv01}
\sup_{0\leq t\leq T}\left( \int \frac{1}{2}\rho |\mathbf{u}|^2 + \frac{P}{\ga-1} dx \right) 
+ \int_0^T \int (2\mu + \lam(\n))  (\div \mathbf{u})^2 + |\na \mathbf{u}|^2 dxdt
\le C.
\ea\ee
\end{lemma}
\begin{proof}
Multiplying $(\ref{ns})_2$ by $\mathbf{u}$ and integrating by parts over $\OM$,
we derive from $(\ref{ns})_1$ and the boundary condition (\ref{i3}) that
\be\la{kv11}\ba
& \frac{d}{dt} \left( \int \frac{1}{2}\rho |\mathbf{u}|^2 + \frac{P}{\ga-1} dx \right)
+ \int (2\mu + \lam(\n)) (\div \mathbf{u})^2 dx + \mu \int |\curl \mathbf{u}|^2 dx \\
& + \mu \int_{\p \OM} \mathbf{u} \cdot K \cdot \mathbf{u} ds = 0,
\ea\ee
where we have used the following fact:
\be\ba\nonumber
\Delta \mathbf{u} = \na \div \mathbf{u} - \na \times \curl \mathbf{u}.
\ea\ee
Combining (\ref{kv11}) with Lemma \ref{dltdc2} yields
\be\ba\nonumber
& \frac{d}{dt} \left( \int \frac{1}{2}\rho |\mathbf{u}|^2 + \frac{P}{\ga-1} dx \right)
+ \int \lam(\n) (\div \mathbf{u})^2 dx + 2\mu \int | D(\mathbf{u}) |^2 dx \\
& + \mu \int_{\p \OM} \mathbf{u} \cdot ( K+2D(n) ) \cdot \mathbf{u} ds = 0,
\ea\ee
which together with Lemma \ref{dltdc1} implies
\be\la{kv14}\ba
& \frac{d}{dt} \left( \int \frac{1}{2}\rho |\mathbf{u}|^2 + \frac{P}{\ga-1} dx \right)
+ \int \lam(\n) (\div \mathbf{u})^2 dx + \frac{\mu}{\Lambda} \int |\na \mathbf{u}|^2 dx
\le 0.
\ea\ee
Integrating (\ref{kv14}) over $(0,T)$, we obtain (\ref{kv01}) and complete the proof of Lemma \ref{l1}.
\end{proof}

\begin{lemma}\la{l2}
Assume that $(\n,\mathbf{u})$ is the strong solution of $(\ref{ns})$ satisfying
the boundary conditions $(\ref{i3})$.
We define
\be\la{yxt01}\ba
F=(2\mu+\lam)\div \mathbf{u} - P,
\ea\ee
which admits the following decomposition:
\be\la{yxt02}\ba
F-\ol{F}=\frac{\p}{\p t}\tilde{F}_1+F_2+F_3.
\ea\ee
Furthermore, for any $1<p<\infty$, there exists a positive constant $C$ depending only on $p$ and $K$ such that
\be\la{yxt03}\ba
\| \tilde{F}_1 \|_{W^{1,p}} \le C \| \n \mathbf{u} \|_{L^p}, \ 
\| F_2 \|_{L^{p}} \le C \| \n \mathbf{u}\otimes \mathbf{u} \|_{L^p}, \ 
\| F_3 \|_{W^{1,p}} \le C \| \na \mathbf{u} \|_{L^p}.
\ea\ee
\end{lemma}

\begin{proof}
First, we consider the Neumann problem
\be\la{yxt1}\ba
\begin{cases}
\Delta{\tilde{F}_1} = \div (\n \mathbf{u}) \ \   & \text{ in } \OM, \\
\int_\OM \tilde{F}_1 =0, \ \frac{\p \tilde{F}_1}{\p n}=0 \ \  & \text{ on } \p \OM.
\end{cases}
\ea\ee

By the boundary condition $\mathbf{u} \cdot n = 0$ on $\p \OM$,
we deduce from \cite[Lemma 4.27]{NS} that the system is solvable,
and for any $1<p<\infty$ the solution satisfies
\be\la{yxt2}\ba
\| \tilde{F}_1 \|_{W^{1,p}} \le C \| \n \mathbf{u} \|_{L^p}.
\ea\ee
Defining $F_1 \triangleq \frac{\p}{\p t}\tilde{F}_1$, it follows from (\ref{yxt1}) that $F_1$ satisfies
\be\la{yxt3}\ba
\begin{cases}
\Delta{F_1} = \frac{\p}{\p t} \div (\n \mathbf{u}) \ \   & \text{ in } \OM, \\
\int_\OM F_1 =0, \ \frac{\p F_1}{\p n}=0 \ \  & \text{ on } \p \OM.
\end{cases}
\ea\ee
Next, let $F_2$ be the solution to the boundary value problem:
\be\la{yxt4}\ba
\begin{cases}
\Delta{F_2} = \div \div ( \n \mathbf{u}\otimes \mathbf{u} ) \ \   & \text{ in } \OM, \\
\int_\OM F_2 =0, \ \frac{\p F_2}{\p n}= \div ( \n \mathbf{u}\otimes \mathbf{u} ) \cdot n \ \  & \text{ on } \p \OM.
\end{cases}
\ea\ee
We now estimate $F_2$ using the method in \cite[Appendix II]{FLW}.
For any $g \in C_0^\infty(\OM)$, let $\varphi$ solve the Neumann problem:
\be\ba\nonumber
\begin{cases}
\Delta{\varphi} = g-\ol{g} \ \   & \text{ in } \OM, \\
\frac{\p \varphi}{\p n} = 0 \ \  & \text{ on } \p \OM.
\end{cases}
\ea\ee
The condition $\int_{\OM} (g-\ol{g}) dx = 0$ ensures the solvability of this system,
and by the standard $L^p$ elliptic estimate (see \cite{GT}), for any $1<p<\infty$, we have
\be\la{yxt6}\ba
\| \na^2 \varphi \|_{L^p} \le C \| g \|_{L^p}.
\ea\ee
Note that $\int_\OM F_2 dx=0$ gives
\be\la{yxt7}\ba
\int F_2 \cdot g dx = \int F_2 (g-\ol{g}) dx = \int F_2 \Delta \varphi dx
= -\int \na F_2 \cdot \na \varphi dx,
\ea\ee
where the boundary term vanishes due to $\na \varphi \cdot n =0$ on $\p \OM$.

On the other hand, by virtue of (\ref{yxt4}) and the boundary condition $\mathbf{u} \cdot n=0$ on $\p \OM$, we have
\be\ba\nonumber
\int \na F_2 \cdot \na \varphi dx
= \int \div ( \n \mathbf{u}\otimes \mathbf{u} ) \cdot \na \varphi dx
= \int ( \n \mathbf{u}\otimes \mathbf{u} ) : \na^2 \varphi dx.
\ea\ee
Combining this with (\ref{yxt6}), (\ref{yxt7}), and H\"older's inequality, we derive
\be\ba\nonumber
\left| \int F_2 \cdot g dx \right|
= \left| \int ( \n \mathbf{u}\otimes \mathbf{u} ) : \na^2 \varphi dx \right|
\le C \| \n \mathbf{u}\otimes \mathbf{u} \|_{L^p} \| g \|_{L^{\frac{p}{p-1}}},
\ea\ee
which implies
\be\la{yxt10}\ba
\| F_2 \|_{L^{p}} \le C \| \n \mathbf{u}\otimes \mathbf{u} \|_{L^p}.
\ea\ee
Furthermore, from $(\ref{ns})_2$ and (\ref{yxt01}), we conclude that
\be\la{yxtx1}\ba
\rho \dot{\mathbf{u}} = \na F - \mu \na \times \curl \mathbf{u}.
\ea\ee
Defining $(K \mathbf{u})^\perp \triangleq - (K \mathbf{u}) \times n$
and applying integration by parts to any $\eta \in C^{\infty}(\OM)$, we find that
\be\la{yxtx2}\ba
& \int \na \times \curl \mathbf{u} \cdot \na \eta dx \\
& = \int \na \times (\curl \mathbf{u} + (K \mathbf{u})^\perp) \cdot \na \eta dx
- \int \na \times (K \mathbf{u})^\perp \cdot \na \eta dx \\
& = - \int \na \times (K \mathbf{u})^\perp \cdot \na \eta dx,
\ea\ee
where we have used $(\curl \mathbf{u} + (K \mathbf{u})^\perp) \times n=0$ on $\p \OM$,
due to (\ref{i3}).

The combination of (\ref{yxtx1}) and (\ref{yxtx2}) yields that for any $\eta \in C^{\infty}(\OM)$
\be\ba\nonumber
\int \na F \cdot \na \eta dx
= \int \left( \rho \dot{\mathbf{u}} -\mu \na \times (K \mathbf{u})^\perp \right) \cdot \na \eta dx,
\ea\ee
which shows that $F$ satisfies the following elliptic equation:
\be\la{yxt11}\ba
\begin{cases}
\Delta F=\div \left( \rho \dot{\mathbf{u}} -\mu \na \times (K \mathbf{u})^\perp \right) & \mathrm{in}\, \,  \OM, \\
\frac {\p F}{\p n}= \left( \rho \dot{\mathbf{u}} - \mu \na \times (K \mathbf{u})^\perp \right) \cdot n &\mathrm{on}\, \,  \p \OM.
\end{cases}
\ea\ee
Finally, we set
\be\la{yxt12}\ba
F_3 \triangleq F-\ol{F}-F_1-F_2.
\ea\ee
From (\ref{yxt3}), (\ref{yxt4}), (\ref{yxt11}), and the boundary condition $\mathbf{u} \cdot n=0$ on $\p \OM$,
we deduce that $F_3$ satisfies
\be\ba\nonumber
\begin{cases}
\Delta F_3=-\mu \div \left( \na \times (K \mathbf{u})^\perp \right) & \mathrm{in}\, \,  \OM, \\
\int_\OM F_3 =0,\ \frac {\p F_3}{\p n}= -\mu \left( \na \times (K \mathbf{u})^\perp \right) \cdot n &\mathrm{on}\, \,  \p \OM.
\end{cases}
\ea\ee
By the standard elliptic estimate (see \cite{NS}), we obtain for any $1<p<\infty$ that
\be\ba\nonumber
\| F_3 \|_{W^{1,p}} \le C \| \na \mathbf{u} \|_{L^p}.
\ea\ee
This combined with (\ref{yxt2}), (\ref{yxt10}), and (\ref{yxt12}) gives
(\ref{yxt02}) and (\ref{yxt03}),
thus completing the proof of Lemma \ref{l2}.
\end{proof}

Building upon the decomposition of $F$, we now establish the $L^\infty(0,T;L^p)$ estimate of the density.
Using the definition of $F$, we rewrite $(\ref{ns})_2$ as
\be\ba\nonumber
\frac{d}{dt} \theta(\n) + P(\n) = -(F-\ol{F}) - \ol{F},
\ea\ee
where $\theta(\n) = 2\mu \log\n + \frac{1}{\beta} \n^\beta$.
By applying (\ref{yxt02}), we obtain
\be\ba\nonumber
\frac{d}{dt} \left( \theta(\n) + \tilde{F}_1 \right) + P(\n)
= \mathbf{u} \cdot \na \tilde{F}_1 - F_2 - F_3 - \ol{F}.
\ea\ee

With the help of (\ref{kv01}), (\ref{yxt03}) and Lemma \ref{ewgn},
along with arguments analogous to those in \cite[Corollary 3.1 and Proposition 3.3]{FLW},
we derive the following time-uniform estimates:

\begin{lemma}\la{l3}
Let $g_{+} \triangleq \max\{ g,0 \}$, then for any $2 \le p <\infty$,
there exist positive constants $C$ and $M$ depending only on
$p$, $\mu$, $\ga$, $\beta$, $\| \n_0 \|_{L^\infty}$, $\| \mathbf{u}_0 \|_{H^1}$, and $K$, such that
\be\la{yzgj1}\ba
\sup_{0\le t\le T} \| \n \|_{L^p} + \int_0^T \int_{\OM} (\n-M)^p_{+} dxdt \le C,
\ea\ee
\be\la{yzgj2}\ba
\int_0^T \int_{\OM} (\n+1)^{\ga-1} (\n-\ol{\n})^2 dx dt \le C.
\ea\ee
\end{lemma}

\begin{lemma}\la{l4}
There exists a positive constant $C$ depending only on $\mu$, $\ga$, $\beta$, $K$, $\|\n_0\|_{L^\infty}$,
and $\|\mathbf{u}_0\|_{H^1}$, such that
\be\la{kv04}\ba
\sup_{0\le t\le T}\int \n |\mathbf{u}|^{2+\nu} dx \le C,
\ea\ee
where
\be\la{kv004}\ba
\nu \triangleq R_T^{-\frac{\beta}{2}} \nu_0,
\ea\ee
for some suitably small generic constant $\nu_0 \in (0,1)$ depending only on $\mu$ and $\ga$.
\end{lemma}
\begin{proof}
First, multiplying $(\ref{ns})_2$ by $(2+\nu)|\mathbf{u}|^\nu \mathbf{u}$ and integrating over $ \OM$, we derive
\be\la{kv41}\ba 
& \frac{1}{(2+\nu)} \frac{d}{dt}\int \n |\mathbf{u}|^{2+\nu} dx
+ \int |\mathbf{u}|^\nu \left(\mu |\curl \mathbf{u}|^2
+ (2\mu+\lam) (\div \mathbf{u})^2 \right) dx
+ \mu \int_{\p \OM} \mathbf{u} \cdot K \cdot \mathbf{u} |\mathbf{u}|^\nu dS \\
& \le C \nu  \int  \left( (2\mu+\lam) |\div \mathbf{u}|+\mu |\curl \mathbf{u}| \right)  |\mathbf{u}|^\nu |\na \mathbf{u}| dx
+ C \int |\n^\ga - \ol{\n}^\ga| |\mathbf{u}|^\nu |\na \mathbf{u}|dx \\
& \triangleq I_1+I_2.
\ea\ee
It follows from (\ref{wdc1}) and Young's inequality that
\be\la{kv42}\ba
I_1 & \le \frac{1}{2} \int |\mathbf{u}|^\nu \left(\mu |\curl \mathbf{u}|^2
+ (2\mu+\lam) (\div \mathbf{u})^2 \right) dx
+ \frac{C\nu_0^2}{2} \int |\mathbf{u}|^\nu |\na \mathbf{u}|^2 dx \\
& \le \frac{1+\hat{C}\nu_0^2}{2} \int |\mathbf{u}|^\nu \left(\mu |\curl \mathbf{u}|^2
+ (2\mu+\lam) (\div \mathbf{u})^2 \right) dx
+ \hat{C} \nu_0^2 \mu \int_{\p \OM} \mathbf{u} \cdot K \cdot \mathbf{u} |\mathbf{u}|^\nu ds,
\ea\ee
provided $\nu \in (0,\hat{\nu})$, where $\hat{C}$ depends only on $\mu$.

Then, when $\nu < \frac{\ga-1}{\ga+1}$, for $s$ satisfying
$\frac{1}{s}=\frac{1-\nu}{2}-\frac{1}{\ga+1}$, by applying
Young's and Poincar\'e's inequalities, we obtain
\be\la{kv43}\ba
I_2 & \le C \int (\n^{\ga-1}+1) |\n-\ol{\n}| |\mathbf{u}|^\nu |\na \mathbf{u}| dx \\
& \le C \int \left( (\n-M)_{+}^{\ga-1} + 1 \right) |\n-\ol{\n}|
|\mathbf{u}|^\nu |\na \mathbf{u}| dx \\
& \le C \left( \int (\n-M)_{+}^{s(\ga-1)} dx
+ \int \left( |\n-\ol{\n}|^{\ga+1} + |\n-\ol{\n}|^{\frac{2}{1-\nu}} \right) dx
+ \int |\na \mathbf{u}|^2 dx \right) \\
& \le C \int (\n-M)_{+}^{s(\ga-1)} dx + C A^2_1,
\ea\ee
where in the last inequality we have used the following estimate:
\be\ba\nonumber
|\n-\ol{\n}|^{\frac{2}{1-\nu}} \le C(\n+1)^{\frac{2\nu}{1-\nu}} (\n-\ol{\n})^2
\le C(\n+1)^{\ga-1} (\n-\ol{\n})^2,
\ea\ee
due to $\nu < \frac{\ga-1}{\ga+1}$.

Substituting (\ref{kv42}) and (\ref{kv43}) into (\ref{kv41}), and choosing
$\nu_0 < \min \left\{\hat{\nu},\frac{1}{\sqrt{2\hat{C}}},\frac{\ga-1}{\ga+1} \right\}$ yields
\be\la{kv44}\ba 
\frac{d}{dt}\int \n |\mathbf{u}|^{2+\nu}dx
\le C \int (\n-M)_{+}^{s(\ga-1)} dx + C A^2_1.
\ea\ee

Therefore, integrating (\ref{kv44}) over $(0,T)$ and using (\ref{kv01}), 
(\ref{yzgj1}), and (\ref{yzgj2}), we arrive at (\ref{kv04}) and finish the proof of Lemma \ref{l4}.
\end{proof}

For $2<p<\infty$, the following estimate of $\| \na \mathbf{u} \|_{L^p}$ will be frequently used and is crucial in the subsequent estimates.
\begin{lemma}\la{l5}
For any $2 < p<\infty$ and $\ep \in (0,1)$,
there exists a positive constant $C$ depending only on $\mu$, $\ga$, $\ep$, $p$, and $\beta$, such that
\be\la{kv05}\ba
\| \nabla \mathbf{u} \|_{L^{p}}
& \le C R^{\frac{1}{2}-\frac{1}{p}+\ep}_T (1+A_1)^{\frac{2}{p}}
(1+A_1+A_2)^{1-\frac{2}{p}}.
\ea\ee
Moreover, when $p<\frac{2(\ga+1)}{\ga}$ and $\ga<2\beta$, we have
\be\la{kv005}\ba
\| \nabla \mathbf{u} \|_{L^{p}}
& \le C R^{\frac{1}{2}-\frac{1}{p}+\ep}_T A_1^{\frac{2}{p}} (1+A_1+A_2)^{1-\frac{2}{p}}.
\ea\ee
\end{lemma}
\begin{proof}
First, choosing $\mathbf{f}=\mathbf{u}$ and $\mathbf{f}=\curl \mathbf{u}$ in Lemma \ref{ewgn}, respectively, and applying Poincar\'e's inequality, we obtain
\be\la{kv51}\ba
\| \mathbf{u} \|_{L^p}
\le C \| \mathbf{u} \|^{\frac{2}{p}}_{L^2}
\| \mathbf{u} \|^{1-\frac{2}{p}}_{H^1}
\le C \| \na \mathbf{u} \|_{L^2},
\ea\ee
and
\be\la{kv52}\ba
\| \curl \mathbf{u} \|_{L^p}
\le C \| \curl \mathbf{u} \|^{\frac{2}{p}}_{L^2}
\| \curl \mathbf{u} \|^{1-\frac{2}{p}}_{H^1}.
\ea\ee
In addition, we define the effective viscous flux $G$ by
\be\ba\la{gw}
G \triangleq (2\mu + \lam)\div \mathbf{u} - (P-P(\ol{\n})),
\ea\ee
and take $g=G$ in Lemma \ref{ewgn} to arrive at
\be\la{kv53}\ba
\| G \|_{L^p}
\le C \| G \|^{\frac{2}{p}}_{L^2} \| G \|^{1-\frac{2}{p}}_{H^1}.
\ea\ee
By virtue of (\ref{gw}), we rewrite $(\ref{ns})_2$ as 
\be\la{kv54}\ba
\n\dot{\mathbf{u}}= \na G - \mu \na\times \curl \mathbf{u},
\ea\ee
which together with the boundary conditions (\ref{i3}) implies that $G$ satisfies the following elliptic equation:
\be\la{kv55}\ba
\begin{cases}
\Delta G=\div \left( \rho \dot{\mathbf{u}} -\mu \na \times (K \mathbf{u})^\perp \right) & \mathrm{in}\, \,  \OM, \\
\frac {\p G}{\p n}= \left( \rho \dot{\mathbf{u}} - \mu \na \times (K \mathbf{u})^\perp \right) \cdot n &\mathrm{on}\, \,  \p \OM.
\end{cases}
\ea\ee

By the standard $L^p$ estimate of elliptic equations as stated in \cite[Lemma 4.27]{NS}, we obtain that for any integer $k \ge 0$ and $1<p<\infty$,
\be\la{kv56}\ba
\| \na G \|_{W^{k,p}} \le C \left( \| \n \dot{\mathbf{u}}\|_{W^{k,p}}
+ \| \na \times (K \mathbf{u})^\perp \|_{W^{k,p}} \right),
\ea\ee
where $C$ depends only on $\mu$, $p$, $k$, and $\OM$.

Note that $(\curl \mathbf{u} + (K \mathbf{u})^\perp) \times n = 0$ on $\p \OM$ and $\div(\na \times \curl \mathbf{u}) = 0$,
and combining this with (\ref{kv54}), (\ref{kv56}) and Lemma \ref{dc}, we derive
\be\la{kv57}\ba
\| \na \curl \mathbf{u} \|_{W^{k,p}}
\le C \left( \| \n \dot{\mathbf{u}}\|_{W^{k,p}}
+ \| \na (K \mathbf{u})^\perp \|_{W^{k,p}} + \| \na \mathbf{u} \|_{L^p} \right).
\ea\ee
In particular, (\ref{kv56}), (\ref{kv57}) and Poincar\'e's inequality lead to
\be\la{kv58}\ba
\| G \|_{H^1} + \| \curl \mathbf{u} \|_{H^1}
& \le C \left( \| \rho \dot{\mathbf{u}} \|_{L^2} + \| \na \mathbf{u} \|_{L^2} \right) + C |\ol{G}| \\
& \le C R^{1/2}_T A_2 + C A_1,
\ea\ee
where in the last inequality we have used the following estimate:
\be\ba\nonumber
\left| \int_\OM G dx \right|
= \left| \int_\OM \lam(\n) \div \mathbf{u} dx \right|
\le C A_1.
\ea\ee
Furthermore, (\ref{yzgj1}) and H\"older's inequality ensure that
\be\la{kv59}\ba
\| G \|^2_{L^2} \le C (1+R^\beta_T A^2_1), \quad
\left\| \frac{G}{2\mu+\lam} \right\|^2_{L^2} \le C(1+A^2_1).
\ea\ee
The combination of (\ref{dc1}), (\ref{kv51}), (\ref{kv52}), (\ref{kv53}), and (\ref{kv59}) implies that
\be\la{kv510}\ba
\| \na \mathbf{u} \|_{L^p}
& \le C \left(\| \div \mathbf{u} \|_{L^p}
+ \| \curl \mathbf{u} \|_{L^p} + \| \mathbf{u} \|_{L^p} \right) \\
& \le C \left\| \frac{G}{2\mu+\lam} \right\|_{L^p}
+ C \left\| \frac{P-P(\ol{\n})}{2\mu+\lam} \right\|_{L^p}
+ C \| \curl \mathbf{u} \|^{\frac{2}{p}}_{L^2} \| \curl \mathbf{u} \|^{1-\frac{2}{p}}_{H^1} + C A_1 \\
& \le C \left\| \frac{G}{2\mu+\lam} \right\|^{\frac{2}{p}-\ep}_{L^2}
\| G \|_{L^{\frac{2(1+\ep)p-4}{p\ep}}}^{-\frac{2}{p}+1+\ep}
+ C \left( A^{\frac{2}{p}}_1 \| \curl \mathbf{u} \|^{1-\frac{2}{p}}_{H^1}
+ A_1 + 1 \right) \\
& \le C (1+A_1)^{\frac{2}{p}-\ep} \| G \|^{\ep}_{L^2} \| G \|^{1-\frac{2}{p}}_{H^1}
+ C \left( A^{\frac{2}{p}}_1 \| \curl \mathbf{u} \|^{1-\frac{2}{p}}_{H^1}
+ A_1 + 1 \right) \\
& \le C R^{\frac{\beta \ep}{2}}_T (1+A_1)^{\frac{2}{p}}
\left( \| G \|_{H^1} + \| \curl \mathbf{u} \|_{H^1} \right)^{1-\frac{2}{p}}
+ C (1+A_1),
\ea\ee
which together with (\ref{kv58}) gives (\ref{kv05}).

Finally, it remains to prove (\ref{kv005}).
Observe that when $p<\frac{2(\ga+1)}{\ga}$ and $\ga<2\beta$, we have $p(\ga-\beta)-2<(\ga-1)$, which yields
\be\la{kv511}\ba
\left\| \frac{P-P(\ol{\n})}{2\mu+\lam} \right\|^p_{L^p}
\le C \int (\n+1)^{p(\ga-\beta)-2} (\n-\ol{\n})^2 dx \le C A^2_1.
\ea\ee
By applying (\ref{gw}) and choosing $p=2$ in (\ref{kv511}), we obtain
\be\la{kv512}\ba
\left\| \frac{G}{2\mu+\lam} \right\|^2_{L^2}
\le C A^2_1 + \left\| \frac{P-P(\ol{\n})}{2\mu+\lam} \right\|^2_{L^2}
\le C A^2_1.
\ea\ee
In addition, Cauchy's inequality gives
\be\la{kv513}\ba
\| G \|^2_{L^2} \le C R^\beta_T A^2_1 + C \| P-P(\ol{\n}) \|^2_{L^2}
\le C R^{\beta+\ga}_T A^2_1.
\ea\ee
Similar to (\ref{kv510}), in view of (\ref{kv511}), (\ref{kv512}), and (\ref{kv513}), we arrive at
\be\ba\nonumber
\| \na \mathbf{u} \|_{L^p}
& \le C \left\| \frac{G}{2\mu+\lam} \right\|_{L^p}
+ C \left\| \frac{P-P(\ol{\n})}{2\mu+\lam} \right\|_{L^p}
+ C \| \curl \mathbf{u} \|^{\frac{2}{p}}_{L^2} \| \curl \mathbf{u} \|^{1-\frac{2}{p}}_{H^1} + C A_1 \\
& \le C A_1^{\frac{2}{p}-\ep} \| G \|^{\ep}_{L^2} \| G \|^{1-\frac{2}{p}}_{H^1}
+ C \left( A^{\frac{2}{p}}_1 \| \curl \mathbf{u} \|^{1-\frac{2}{p}}_{H^1}
+ A_1 + A^{\frac{2}{p}}_1 \right) \\
& \le C R^{\frac{ (\beta+\ga) \ep}{2}}_T A_1^{\frac{2}{p}}
\left( \| G \|_{H^1} + \| \curl \mathbf{u} \|_{H^1} \right)^{1-\frac{2}{p}}
+ C \left( A_1 + A^{\frac{2}{p}}_1 \right),
\ea\ee
which along with (\ref{kv58}) implies (\ref{kv005}) and completes the proof of Lemma \ref{l5}.
\end{proof}

\begin{lemma}\la{l6}
For any $\ep \in (0,1)$,
there exists a positive constant $C$ depending only on $\ep$, $\ga$, $\mu$, 
$\beta$, $\|\n_0\|_{L^\infty}$, $\|\mathbf{u}_0\|_{H^1}$, and $K$, such that
\be\la{kv06}\ba
\sup_{0 \le t \le T} \log(e+A^2_1(t)) + \int_0^T \frac{A^2_2(t)}{e+A^2_1(t)} dt
\le C R^{1+\ep}_T.
\ea\ee
\end{lemma}

\begin{proof}
First, direct calculations yield
\be\la{kv61}\ba 
\div \dot{\mathbf{u}}
= \frac{D}{Dt} \left( \frac{G}{2\mu + \lam} \right)
+ \frac{D}{Dt} \left( \frac{P-P(\ol{\n})}{2\mu + \lam} \right) + \mathbf{g}_1,
\ea\ee
and
\be\la{kv62}\ba
\na \times \dot{\mathbf{u}}= \frac{D}{Dt} \curl \mathbf{u} + \mathbf{g}_2,
\ea\ee
where $\mathbf{g}_1$ and $\mathbf{g}_2$ satisfy $|\mathbf{g}_1|+|\mathbf{g}_2| \le C |\na \mathbf{u}|^2$.

Multiplying (\ref{kv54}) by $2 \dot{\mathbf{u}}$ and integrating
the resulting equality over $\OM$,
by (\ref{kv61}) and (\ref{kv62}), we derive
\be\la{kv63}\ba
& \frac{d}{dt} \int \left(\mu |\curl \mathbf{u}|^2 + \frac{G^2}{2\mu + \lam}\right)dx
+ 2 A^2_2 \\
& = \mu \int | \curl \mathbf{u} |^2 \div \mathbf{u} dx
-2 \mu \int \curl \mathbf{u} \cdot \mathbf{g}_2 dx
- 2 \int G \cdot \mathbf{g}_1 dx \\
&\quad - \int \frac{ (\beta-1)\lam - 2\mu }{(2\mu + \lam)^2} G^2\div \mathbf{u} dx
-2\beta \int \frac{ \lam (P-P(\ol{\n})) }{ (2\mu +\lam)^2 } G \div \mathbf{u} dx
+ 2\ga \int \frac{P}{2\mu +\lam} G \div \mathbf{u} dx \\
& \quad + 2 \int_{\p\OM} G \mathbf{u} \cdot \na \mathbf{u} \cdot n ds
-2\mu \int_{\p \OM} \dot{\mathbf{u}} \cdot K \cdot \mathbf{u} ds
= \sum_{i=1}^8I_i.
\ea\ee

We now estimate each $I_i$ as follows:

First, H\"older's inequality gives
\be\la{kv64}\ba
|I_1+I_2+I_3| & \le C \int \left( |G| + |\curl \mathbf{u}| \right) |\na \mathbf{u}|^2 dx \\
& \le C \left( \| G \|_{L^p} + \| \curl \mathbf{u} \|_{L^p} \right)
\| \na \mathbf{u} \|^2_{L^{\frac{2p}{p-1}}}.
\ea\ee
Combining Lemma \ref{ewgn} with (\ref{kv58}) and (\ref{kv513}) leads to
\be\la{kv65}\ba
\| G \|_{L^p} + \| \curl \mathbf{u} \|_{L^p}
& \le \left( \| G \|_{L^2} + \| \curl \mathbf{u} \|_{L^2} \right)^{\frac{2}{p}}
\left( \| G \|_{H^1} + \| \curl \mathbf{u} \|_{H^1} \right)^{1-\frac{2}{p}} \\
& \le C R^{\frac{1}{2}-\frac{1}{p}+\frac{\beta+\ga}{p}}_T A_1^{\frac{2}{p}}
(A_1+A_2)^{1-\frac{2}{p}}.
\ea\ee
On the other hand, we deduce from (\ref{kv05}) and H\"older's inequality that
\be\la{kv66}\ba
\| \na \mathbf{u} \|^2_{L^{\frac{2p}{p-1}}}
& \le \| \na \mathbf{u} \|^{\frac{2(p-3)}{p-2}}_{L^2}
\| \na \mathbf{u} \|^{\frac{2}{p-2}}_{L^p} \\
& \le C R^{\frac{1}{p}+\ep}_T A_1^{\frac{2(p-3)}{p-2}}
\left( (1+A_1)^{\frac{2}{p}} (1+A_1+A_2)^{1-\frac{2}{p}} \right)^{\frac{2}{p-2}}.
\ea\ee

Putting (\ref{kv65}) and (\ref{kv66}) into (\ref{kv64}), applying Young's inequality
and letting $p>4+(\beta+\ga)/\ep$, we arrive at
\be\la{kv67}\ba
|I_1+I_2+I_3| & \le C R^{\frac{1}{2}+\ep}_T
A_1^{\frac{2}{p} + \frac{2(p-3)}{p-2}} (A_1+A_2)^{1-\frac{2}{p}}
\left( (1+A_1)^{\frac{2}{p}} (1+A_1+A_2)^{1-\frac{2}{p}} \right)^{\frac{2}{p-2}}\\
& \le C R^{\frac{1}{2}+\ep}_T (A_1+A^2_1) (A_1+A^2_1+A_2) \\
& \le \frac{1}{8} A^2_2 + C R^{1+\ep}_T (1+A^2_1) A^2_1.
\ea\ee
In addition, by virtue of (\ref{kv58}), (\ref{kv66}), (\ref{kv67}), and Young's inequality, it holds that
\be\la{kv68}\ba
|I_4+I_5+I_6|
& \le C \int \frac{ G^2 |\div \mathbf{u}| }{2\mu+\lam} dx
+ C \int \frac{ P + P(\ol{\n}) }{ 2\mu +\lam } |G| |\div \mathbf{u}| dx \\
& \le C \int |G| (\div \mathbf{u})^2 dx
+ C \int \frac{ P+P(\ol{\n}) }{ 2\mu +\lam } |G| |\div \mathbf{u}| dx \\
& \le C \| G \|_{L^p} \| \na \mathbf{u} \|^2_{L^{\frac{2p}{p-1}}}
+ C \| G \|_{L^4} \| \div \mathbf{u} \|_{L^2} \\
& \le C R^{\frac{1}{2}+\ep}_T (A_1+A^2_1) (A_1+A^2_1+A_2)
+ C A_1 \left( R^{1/2}_T A_2 + A_1 \right) \\
& \le \frac{1}{8} A^2_2 + C R^{1+\ep}_T (1+A^2_1) A^2_1.
\ea\ee
For $I_7$, it follows from (\ref{bjds}), (\ref{kv58}), and Young's inequality that
\be\la{kv610}\ba
|I_7| = 2 \left| \int_{\p\OM} G \mathbf{u} \cdot \na n \cdot \mathbf{u} ds \right|
& \le C \| G \|_{H^1} \| \na \mathbf{u} \|^2_{L^2} \\
& \le C \left( R^{1/2}_T A_2 + A_1 \right) A^2_1 \\
& \le \frac{1}{8} A^2_2 + C R_T A^4_1 + C A^2_1.
\ea\ee
Moreover, by using (\ref{bjds}), (\ref{kv05}), and Poincar\'e's inequality, we derive
\begin{align}\label{kv612}
I_8 & = -2\mu \int_{\p \OM} \dot{\mathbf{u}} \cdot K \cdot \mathbf{u} ds \nonumber \\
& = -\mu \frac{d}{dt} \int_{\p \OM} \mathbf{u} \cdot K \cdot \mathbf{u} ds
-2\mu \int_{\p \OM} \mathbf{u} \cdot \na \mathbf{u} \cdot K \cdot \mathbf{u} ds \nonumber \\
& = -\mu \frac{d}{dt} \int_{\p \OM} \mathbf{u} \cdot K \cdot \mathbf{u} ds
-2\mu \int_{\p \OM} \mathbf{u}^\bot \times n \cdot \nabla \mathbf{u}^i (K^i \cdot \mathbf{u} ) ds \nonumber \\
& = - \mu \frac{d}{dt} \int_{\p \OM} \mathbf{u} \cdot K \cdot \mathbf{u} ds
-2\mu \int_{\p \OM} n \cdot ( \na \mathbf{u}^i \times \mathbf{u}^\bot ) (K^i \cdot \mathbf{u} ) ds \nonumber \\
& = - \mu \frac{d}{dt} \int_{\p \OM} \mathbf{u} \cdot K \cdot \mathbf{u} ds
-2\mu \int \div ( ( \na \mathbf{u}^i \times \mathbf{u}^\bot ) (K^i \cdot \mathbf{u} ) )dx \nonumber \\
& = - \mu \frac{d}{dt} \int_{\p \OM} \mathbf{u} \cdot K \cdot \mathbf{u} ds
+2\mu \int ( \na \mathbf{u}^i \cdot \na \times \mathbf{u}^\bot ) (K^i \cdot \mathbf{u} ) dx \nonumber \\
& \quad -2\mu \int  \na (K^i \cdot \mathbf{u} ) \cdot ( \na \mathbf{u}^i \times \mathbf{u}^\bot ) dx \nonumber \\
& \le - \mu \frac{d}{dt} \int_{\p \OM} \mathbf{u} \cdot K \cdot \mathbf{u} ds
+ C \int |\na \mathbf{u}|^2 |\mathbf{u}| + |\na \mathbf{u}| |\mathbf{u}|^2 dx \nonumber \\
& \le - \mu \frac{d}{dt} \int_{\p \OM} \mathbf{u} \cdot K \cdot \mathbf{u} ds
+ C \| \na \mathbf{u} \|_{L^4} \| \na \mathbf{u}\|^2_{L^2}
+ C \| \na \mathbf{u} \|^3_{L^2} \nonumber \\
& \le - \mu \frac{d}{dt} \int_{\p \OM} \mathbf{u} \cdot K \cdot \mathbf{u} ds
+ C R^{\frac{1}{4}+\ep}_T A^2_1 (1+A_1+A_2) + C A^3_1 \nonumber \\
& \le - \mu \frac{d}{dt} \int_{\p \OM} \mathbf{u} \cdot K \cdot \mathbf{u} ds
+ \frac{1}{8} A^2_2 + C R^{1+\ep}_T (1+A^2_1) A^2_1,
\end{align}
where the symbol $K^i$ denotes the $i$-th row of the matrix $K$ and we have used the following fact:
\be\la{kv613}\ba
\div ( \na \mathbf{u}^i \times \mathbf{u}^\bot ) = -\na \mathbf{u}^i \cdot \na \times \mathbf{u}^\bot.
\ea\ee

Substituting (\ref{kv67})--(\ref{kv612}) into (\ref{kv63}) yields
\be\la{kv614}\ba
\frac{d}{dt} A^2_3 + A^2_2 \le C R^{1+\ep}_T (1+A^2_1) A^2_1,
\ea\ee
where
\be\la{kv615}\ba
A_3^2(t) \triangleq \int \left( \frac{G^2(t)}{2\mu+\lambda}+\mu|\curl \mathbf{u}|^2(t) \right) dx
+ \mu \int_{\p \OM} \mathbf{u} \cdot K \cdot \mathbf{u} ds.
\ea\ee
In addition, we conclude from (\ref{dltdc01}) and (\ref{dltdc02}) that
\be\la{kv616}\ba
\| \na \mathbf{u} \|^2_{L^2}
\le C \left( \| \div \mathbf{u} \|^2_{L^2} + \| \curl \mathbf{u} \|^2_{L^2}
+ \int_{\p \OM} \mathbf{u} \cdot K \cdot \mathbf{u} ds \right),
\ea\ee
which together with (\ref{yzgj1}) and (\ref{kv615}) implies
\be\la{kv617}\ba
\frac{1}{C} (e+A^2_3) \le e + A^2_1 \le C(e+A^2_3).
\ea\ee
Therefore, dividing (\ref{kv614}) by $e+A^2_3$ and applying (\ref{kv617}), we arrive at
\be\la{kv618}\ba
\frac{d}{dt} \log (e+A^2_3) + \frac{A^2_2}{e+A^2_1}
\le C R^{1+\ep}_T A^2_1.
\ea\ee

Integrating (\ref{kv618}) over $(0,T)$ and using
(\ref{kv01}), (\ref{yzgj1}), (\ref{yzgj2}), and (\ref{kv617}),
we obtain (\ref{kv06}) and complete the proof of Lemma \ref{l6}.
\end{proof}

Next, we estimate the effective viscous flux $G$.
Since $G$ is axisymmetric and solves a Neumann boundary problem,
we can exploit this symmetry to reduce the three-dimensional problem to a two-dimensional one.
This reduction allows us to apply the approach in \cite{FLL}, which deals with problems in two-dimensional bounded simply connected domains, to derive the corresponding estimates.

We note that the method in \cite{FLL} relies on the boundary condition $(\mathbf{u} \cdot n)|_{\p \OM}=0$ to cancel out the singularity.
However, in our setting, the periodicity in the $x_3$-direction prevents us from
imposing boundary conditions on the top and bottom surfaces of $\OM$.
The absence of these boundary conditions obstructs direct estimates for $G$ over the entire domain $\OM$.

To overcome this difficulty, we extend $\OM$ periodically in the $x_3$-direction
to a larger domain $\OM_1$, and then establish estimates for $G$ on $\OM$ by working within $\OM_1$.
Specifically, we define
\be\ba\nonumber
\OM_1 \triangleq \{ (x_1,x_2,x_3) \in \rrr: 1<x^2_1+x^2_2<4, -2<x_3<3 \}.
\ea\ee
By the periodicity in $x_3$ and (\ref{kv55}), we obtain that for any $t\in [0,T]$, $G$ satisfies
the following elliptic equation with Neumann boundary conditions:
\be\la{evf1}\ba
\begin{cases}
\Delta G=\div \left( \rho \dot{\mathbf{u}} \right) & \mathrm{in}\, \,  \OM_1, \\
\frac {\p G}{\p n}= \left( \rho \dot{\mathbf{u}} - \mu \na \times (K \mathbf{u})^\perp \right) \cdot n &\mathrm{on}\, \,  \p \OM_1.
\end{cases}
\ea\ee

Exploiting the axisymmetry of the problem, we transform the above equation into a two-dimensional form.
Let $\tilde{\Delta} \triangleq \p_{rr}+ \p_{zz}$ and $\tilde{\na} \triangleq (\p_{r}, \p_{z})$.
Direct calculations yield
\be\la{evf2}\ba
\begin{cases}
\tilde{\Delta} G=\div \left( \rho \dot{\mathbf{u}} \right) - \frac{1}{r}\p_r G & \mathrm{in}\, \,  D_1, \\
\tilde{\na} G \cdot \tilde{n} = \left( \rho \dot{\mathbf{u}} - \mu \na \times (K \mathbf{u})^\perp \right) \cdot n &\mathrm{on}\, \,  \p D_1,
\end{cases}
\ea\ee
where $D_1 \triangleq \{(r,z) \in \rr :1<r<2, -2<z<3 \}$, and $\tilde{n}$ denotes the unit outer normal vector of the boundary $\p D_1$.
Note that the Green's function $N(x,y)$ for the Neumann problem (see \cite{STT})
on the two-dimensional unit disc $\mathbb{D}$ is given by
\be\ba\nonumber
N(x,y)=-\frac{1}{2\pi}\bigg(\log|x-y|+\log\left||x|y-\frac{x}{|x|}\right|\bigg).
\ea\ee

Moreover, by the Riemann mapping theorem (see \cite{SES}), there exists a conformal mapping
$\varphi=(\varphi_1, \varphi_2):\overline{D_1}\rightarrow\overline{\mathbb{D}}$.
We define the pull back Green's function $\widetilde{N}(x,y)$ on $D_1$ as follows:
\be\ba\nonumber
\widetilde{N}(x,\, y) \triangleq N\big(\varphi(x),\varphi(y)\big) \ \ \mathrm{for}\ x,y\in D_1.
\ea\ee

Before deriving the estimates for $G$, we first introduce some notation.
In the following Lemmas \ref{l7}--\ref{l9}, unless otherwise specified,
for any $\mathbf{x}=(x_1,x_2,x_3) \in \OM_1$,
let $x=(r_\mathbf{x},z_\mathbf{x}) \in D_1$ denote the corresponding
two-dimensional coordinates under the axisymmetric transformation,
where $r_\mathbf{x}=\sqrt{x^2_1+x^2_2}$ and $z_\mathbf{x}=x_3$.
We also set $u_1 \triangleq u_r$, $u_2 \triangleq u_z$.

Based on the above definitions and notation, we now establish the estimates for $G$.

\begin{lemma}\label{l7}
Assume that $G\in C\big([0,T];C^1(\overline{\OM_1})\cap C^2(\OM_1)\big)$ satisfies the equation $(\ref{evf1})$.
Then for any $\mathbf{x}\in \OM$, there exists a positive constant $C$ depending only on
$\ga$, $\mu$, $\beta$, $\|\n_0\|_{L^\infty}$, $\|\mathbf{u}_0\|_{H^1}$, and $K$, such that
\be\la{kv701}\ba
-G(\mathbf{x},t)
& \le \frac{D}{Dt} \psi (\mathbf{x},t) 
+ C \left( \| \sqrt{\n} \dot{\mathbf{u}} \|_{L^2}
+ \| \na \mathbf{u} \|^2_{L^2} + \| G \|_{H^1}
+ \| \na \mathbf{u} \|_{L^4} \right) - J,
\ea\ee
where
\be\la{kv702}\ba
\psi \triangleq \int_{D_1}  \left( \p_{y_i} \widetilde{N}(x,y) \n u_i(y) \right) dy,
\ea\ee
and
\be\la{kv703}\ba
J \triangleq \int_{D_1} \left( \p_{x_i} \p_{y_j} \widetilde{N}(x,y) u_i(x)
+ \p_{y_i} \p_{y_j} \widetilde{N}(x,y) u_i(y) \right) \n u_j(y) dy.
\ea\ee
\end{lemma}
\begin{proof}
First, since $G$ satisfies equation (\ref{evf2}), it follows from \cite[Lemma 3.7]{FLL} that for $x=(r_\mathbf{x},z_\mathbf{x}) \in D \subset D_1$,
\be\la{kv71}\ba
-G(x,t)
=& -\int_{D_1} \widetilde{N}(x,\, y)  \left( \div \left( \rho \dot{\mathbf{u}} \right) - \frac{1}{r_y}\p_r G \right) \, dy
- \int_{\partial D_1} \frac{\partial \widetilde{N}}{\partial n}(x,\, y) G(y) dS_y \\
&+ \int_{\partial D_1} \widetilde{N}(x,y) \left( \rho \dot{\mathbf{u}} - \mu \na \times (K \mathbf{u})^\perp \right) \cdot n dS_y.
\ea\ee
Next, we estimate each term on the right-hand side of (\ref{kv71}).

From \cite[Lemma 3.6]{FLL}, we conclude that for any $x\in D_1$, $y\in \p D_1$,
\be\ba\la{kv72}
\frac{\partial \widetilde{N}}{\partial n}(x,\, y)=-\frac{1}{2 \pi} |\na \varphi_1 (y)|.
\ea\ee
Moreover, for any $x,y \in D_1$, direct calculation shows that
\be\ba\nonumber
|\varphi(x) - \varphi(y)| \le 4 \left| |\varphi(x)| \varphi(y) - \frac{\varphi(x)}{ |\varphi(x)| }\right|,
\ea\ee
which implies
\be\la{kv74}\ba
|\widetilde{N}(x,y)| \le C \left( 1 + | \log |x-y| | \right).
\ea\ee
By applying (\ref{kv72}), (\ref{kv74}) and H\"older's inequality, we obtain
\be\la{kv75}\ba
& \int_{D_1} \frac{1}{r_y} \widetilde{N}(x,\, y) \p_r G \, dy
- \int_{\partial D_1} \frac{\partial \widetilde{N}}{\partial n}(x,\, y) G(y) dS_y \\
& \le C \| G \|_{H^1(D_1)} \le C \| G \|_{H^1(\OM)}.
\ea\ee

Before proceeding to the next estimate, we first show that for any $x \in D$ and $y \in D_1$, $|\varphi(x)-\varphi(y)|$ is equivalent to $|x-y|$.

Note that the domain $D_1$ has corners, so that the derivative of the conformal map $\varphi$ tends to zero near these corners.
For this reason, we partition $D_1$ into a domain near the corners and another domain bounded away from them.
Specifically, we define:
\be\nonumber\ba
D^{*}_1 \triangleq \{ (r,z) \in \rr : 1<r<4, -1<z<2 \}, \quad D^{+}_1 \triangleq D_1 \setminus D^{*}_1.
\ea\ee

On the one hand, for any $x \in D$ and $y \in D^{*}_1$,
it follows from \cite{SES} that there exists a constant $c_0>0$, depending only on $D_1$, such that
\be\la{qyfj2}\ba
\frac{1}{c_0} |x-y| \le |\varphi(x) - \varphi(y)| \le c_0 |x-y|.
\ea\ee

On the other hand, for any $x \in D$ and $y \in D^{+}_1$, we have $|x-y| \ge 1$.
By the continuity of $\varphi$,
there exists a constant $c_1 \in (0,1)$, depending only on $D_1$, such that $|\varphi(x) - \varphi(y)| \ge c_1$.
Combining this with (\ref{qyfj2}) implies the existence of a constant $c_2>0$ such that
\be\la{qyfj3}\ba
\frac{1}{c_2} |x-y| \le |\varphi(x) - \varphi(y)| \le c_2 |x-y|, \  \text{ for any } x \in D, y \in D_1.
\ea\ee

Then, for any $\mathbf{y} =(y_1,y_2,y_3) \in \ol{\OM_1}$, we set $\hat{y}=\left(r_\mathbf{y},y_3\right) \in \ol{D_1}$ with $r_\mathbf{y}=\sqrt{y_1^2+y_2^2}$,
and define $\hat{N}(\mathbf{x},\mathbf{y}) \triangleq \widetilde{N}(x,\hat{y})$.
Integrating by parts and applying (\ref{kv74}) and Poincar\'e's inequality, we derive
\be\la{kv76}\ba
&  \mu \left| \int_{\partial D_1}  \left( \widetilde{N}(x,y) \na \times (K \mathbf{u})^\perp \right) \cdot n dS_y \right| \\
& = \frac{\mu}{2 \pi} \left| \int_{\partial \OM_1} \frac{1}{r_\mathbf{y}}
\left( \hat{N}(\mathbf{x},\mathbf{y}) \na \times (K \mathbf{u})^\perp \right) \cdot n dS_\mathbf{y} \right| \\
& = \mu \left| \int_{\OM_1} \div
\left( \frac{1}{r_\mathbf{y}} \hat{N}(\mathbf{x},\mathbf{y}) \na \times (K \mathbf{u})^\perp \right) d\mathbf{y} \right| \\
& \le C \int_{\OM_1} \left( |\hat{N}(\mathbf{x},\mathbf{y})| + |\na_{\mathbf{y}} \hat{N}(\mathbf{x},\mathbf{y})| \right)
\left( |\mathbf{u}| + |\na \mathbf{u}| \right) d\mathbf{y} \\
& \le C \int_{\OM_1} \left( 1 + |x-\hat{y}|^{-1} \right)
\left( |\mathbf{u}| + |\na \mathbf{u}| \right) d\mathbf{y} \\
& \le C \| \na \mathbf{u} \|_{L^4},
\ea\ee
where we have used the estimates
$|\na_{\mathbf{y}} \hat{N}(\mathbf{x},\mathbf{y})| \le C |\varphi(x)-\varphi(\hat{y})|^{-1} \le C |x-\hat{y}|^{-1}$, due to (\ref{qyfj3}).

By direct computation, we have
\be\ba\la{kvx77}
\div (\n \mathbf{u} \otimes \mathbf{u})
& = \left( \frac{\p (\n u_r^2)}{\p r} + \frac{\p (\n u_r u_z)}{\p z} + \frac{\n (u_r^2-u_\theta^2)}{r} \right) \mathbf{e}_r \\
& \quad + \left( \frac{\p (\n u_r u_\theta)}{\p r} + \frac{\p (\n u_\theta u_z)}{\p z} + \frac{2 \n u_r u_\theta}{r} \right) \mathbf{e}_\theta \\
& \quad + \left( \frac{\p (\n u_r u_z)}{\p r} + \frac{\p (\n u^2_z)}{\p z} + \frac{\n u_r u_z}{r} \right) \mathbf{e}_z \\
& \triangleq H_r \mathbf{e}_r + H_\theta \mathbf{e}_\theta + H_z \mathbf{e}_z,
\ea\ee
which together with the mass equation $(\ref{ns})_1$ implies
\be\ba\la{kvx78}
\n \dot{\mathbf{u}} = \left( (\n u_r)_t + H_r \right) \mathbf{e}_r
+ \left( (\n u_\theta)_t + H_\theta \right) \mathbf{e}_\theta
+ \left( (\n u_z)_t + H_z \right) \mathbf{e}_z.
\ea\ee
Thus, direct calculation gives
\be\ba\la{kvx79}
\div(\n \dot{\mathbf{u}})
= \frac{1}{r} \frac{\p}{\p r} \big( r (\n u_r)_t + r H_r \big)
+ \frac{\p}{\p z} \big( (\n u_z)_t + H_z \big).
\ea\ee
Integrating by parts and using (\ref{yzgj1}), (\ref{kvx78}), (\ref{kvx79}), and H\"older's inequality, we obtain
\be\la{kvx710}\ba
& - \int_{D_1} \widetilde{N}(x,y)  \div \left( \rho \dot{\mathbf{u}} \right) dy
+ \int_{\partial D_1} \widetilde{N}(x,y) (\rho \dot{\mathbf{u}} \cdot n) dS_y \\
& = - \int_{D_1} \widetilde{N}(x,y) \left( \frac{1}{r} \frac{\p}{\p r} \big( r (\n u_r)_t + r H_r \big) + \frac{\p}{\p z} \big( (\n u_z)_t + H_z \big) \right) dy \\
& \quad + \int_{\partial D_1} \widetilde{N}(x,y) \left( (\n u_r)_t + H_r \right) dS_y \\
& = \int_{D_1}  \left( \p_{y_1} \widetilde{N}(x,y) \left( (\n u_r)_t + H_r \right)
+ \p_{y_2} \widetilde{N}(x,y) \left( (\n u_z)_t + H_z \right) \right) dy \\
& \quad - \int_{D_1} \left( \frac{1}{r} \widetilde{N}(x,y) \left( (\n u_r)_t + H_r \right) \right) dy \\
& \le C \| \sqrt{\n} \dot{\mathbf{u}} \|_{L^2}
+ \int_{D_1}  \left( \p_{y_1} \widetilde{N}(x,y) \left( (\n u_r)_t + H_r \right)
+ \p_{y_2} \widetilde{N}(x,y) \left( (\n u_z)_t + H_z \right) \right) dy.
\ea\ee

We now estimate the second term on the last line of (\ref{kvx710}).
From the definition of the material derivative and the axisymmetry of the solution, we deduce that
\be\ba\la{kvx711}
& \int_{D_1}  \left( \p_{y_1} \widetilde{N}(x,y) (\n u_r)_t
+ \p_{y_2} \widetilde{N}(x,y) (\n u_z)_t \right) dy \\
& = \frac{d}{dt} \left( \int_{D_1} \left( \p_{y_1} \widetilde{N}(x,y) \n u_r
+ \p_{y_2} \widetilde{N}(x,y) \n u_z \right) dy \right) \\
& = \frac{D}{Dt} \left( \int_{D_1}  \left( \p_{y_1} \widetilde{N}(x,y) \n u_r
+ \p_{y_2} \widetilde{N}(x,y) \n u_z \right) dy \right) \\
& \quad - \int_{D_1} \left( \p_{x_1} \p_{y_1} \widetilde{N}(x,y) \n u_r(y) u_r(x) + \p_{x_2} \p_{y_1} \widetilde{N}(x,y) \n u_r(y) u_z(x) \right) dy \\
& \quad - \int_{D_1} \left( \p_{x_1} \p_{y_2} \widetilde{N}(x,y) \n u_z(y) u_r(x) + \p_{x_2} \p_{y_2} \widetilde{N}(x,y) \n u_z(y) u_z(x) \right) dy.
\ea\ee
Moreover, integration by parts combined with (\ref{yzgj1}) and Poincar\'e's inequality yields
\be\ba\la{kvx712}
& \int_{D_1}  \left( \p_{y_1} \widetilde{N}(x,y) H_r
+ \p_{y_2} \widetilde{N}(x,y) H_z \right) dy \\
& = \int_{D_1} \left( \p_{y_1} \widetilde{N}(x,y) \left( \frac{\p (\n u_r^2)}{\p r} + \frac{\p (\n u_r u_z)}{\p z} + \frac{\n (u_r^2-u_\theta^2)}{r} \right) \right) dy \\
& \quad + \int_{D_1} \left( \p_{y_2} \widetilde{N}(x,y) \left( \frac{\p (\n u_r u_z)}{\p r} + \frac{\p (\n u^2_z)}{\p z} + \frac{\n u_r u_z}{r} \right) \right) dy \\
& \le - \int_{D_1} \left( \p_{y_1} \p_{y_1} \widetilde{N}(x,y) \n u_r u_r(y) + \p_{y_2} \p_{y_1} \widetilde{N}(x,y) \n u_r u_z(y) \right) dy \\
& \quad - \int_{D_1} \left( \p_{y_1} \p_{y_2} \widetilde{N}(x,y) \n u_z u_r(y) + \p_{y_2} \p_{y_2} \widetilde{N}(x,y) \n u_z u_z(y) \right) dy
+ C \| \na \mathbf{u} \|^2_{L^2}.
\ea\ee
Substituting (\ref{kvx711}) and (\ref{kvx712}) into (\ref{kvx710}) leads to
\be\la{kvx713}\ba
& - \int_{D_1} \widetilde{N}(x,y)  \div \left( \rho \dot{\mathbf{u}} \right) dy
+ \int_{\partial D_1} \widetilde{N}(x,y) (\rho \dot{\mathbf{u}} \cdot n) dS_y \\
& \le C \| \sqrt{\n} \dot{\mathbf{u}} \|_{L^2} + C \| \na \mathbf{u} \|^2_{L^2}
+ \frac{D}{Dt} \left( \int_{D_1}  \left( \p_{y_i} \widetilde{N}(x,y) \n u_i(y) \right) dy \right) - J,
\ea\ee
where
\be\ba\nonumber
J \triangleq \int_{D_1} \left( \p_{x_i} \p_{y_j} \widetilde{N}(x,y) u_i(x)
+ \p_{y_i} \p_{y_j} \widetilde{N}(x,y) u_i(y) \right) \n u_j(y) dy.
\ea\ee
Combining (\ref{kv71}), (\ref{kv75}), (\ref{kv76}), and (\ref{kvx713}) implies
\be\ba\nonumber
-G(\mathbf{x},t) = - G(x,t)
& \le \frac{D}{Dt} \left( \int_{D_1}  \left( \p_{y_i} \widetilde{N}(x,y) \n u_i(y) \right) dy \right) + C \| G \|_{H^1} \\
& \quad + C \| \sqrt{\n} \dot{\mathbf{u}} \|_{L^2} + C \| \na \mathbf{u} \|^2_{L^2}
+ C \| \na \mathbf{u} \|_{L^4} - J,
\ea\ee
which gives (\ref{kv701}) and completes the proof of Lemma \ref{l7}.
\end{proof}

\begin{lemma}\la{ll7}
For $J$ as in (\ref{kv703}), there exists a positive constant $C$ depending only on $\ga$, $\mu$, $\beta$, $\|\n_0\|_{L^\infty}$, $\|\mathbf{u}_0\|_{H^1}$, and $K$, such that for any $\mathbf{x} \in \OM$ with $\varphi(x) \ne 0$,
\be\la{kvl701}\ba
|J| & \le C \| \na \mathbf{u} \|^2_{L^2}
+ C \sup_{x \in \ol{D_1}} \left( \sum^{2}_{i,j=1} \int_{D_1} \frac{\left| u_i(x) - u_i(y) \right|}{|x-y|^2} \n |u_j| (y) dy \right).
\ea\ee
\end{lemma}
\begin{proof}
First, we rewrite $J$ as
\be\la{kv79}\ba
J& = \int_{D_1} \p_{x_i} \p_{y_j} \widetilde{N}(x,y)
\left( u_i(x) - u_i(y) \right) \n u_j(y) dy \\
& \quad - \int_{D_1} \Lambda_{i,j}(\varphi (y), \varphi (x)) \n u_i u_j(y) dy \\
& \quad - \int_{D_1} \Lambda _{i,j}(\varphi (y), w(x)) \n u_i u_j(y) dy
\triangleq \sum _{l=1}^3 J_l,
\ea\ee
with
\be\ba\nonumber
\Lambda_{i, j} (\varphi (y), v(x)) \triangleq (\p_{x_i} \p_{y_j}
+\p_{y_i} \p_{y_j}) \log \left| \varphi (y)-v(x) \right|,
w(x) \triangleq \frac{\varphi (x)}{\vert \varphi (x)\vert ^2}.
\ea\ee
For $J_1$, direct calculation shows that
\be\la{kv711}\ba
|J_1| \le C \sum^{2}_{i,j=1} \int_{D_1} \frac{\left| u_i(x) - u_i(y) \right|}{|x-y|^2} \n |u_j| (y) dy.
\ea\ee
Next, to estimate $J_2$ and $J_3$, for $v(x) \in \{\varphi (x), w(x)\}$, we have
\be\la{kv712}\ba
& \Lambda_{i, j}(\varphi (y), v(x)) \\
& \quad = \frac{ (\varphi _k(y)-v_k(x))\partial _j\partial _i\varphi _k(y)}{\vert v(x) -\varphi (y)\vert ^2}+ \frac{\partial _j\varphi _k(y)(\partial _i\varphi _k(y)-\partial _iv_k(x)) }{\vert v(x) -\varphi (y)\vert ^2} \\
& \qquad +2\frac{(v_k(x)-\varphi _k(y))(\partial _iv_k(x)-\partial _i \varphi _k(y))( \varphi _s(y)-v_s(x))\partial _j\varphi _s(y)}{\vert v(x)-\varphi (y)\vert ^4}.
\ea\ee
Consequently, by virtue of (\ref{qyfj3}) and (\ref{kv712}), it holds that
\be\ba\nonumber
\vert \Lambda _{i, j}(\varphi (y), \varphi (x))\vert \le C \vert x-y\vert ^{-1},
\ea\ee
which together with Poincar\'e's inequality implies
\be\la{kv714}\ba
|J_2| \le C \int_{D_1} \frac{\n ( u_r^2 + u_z^2 )}{|x-y|} dy
\le C \| \na \mathbf{u} \|^2_{L^2}.
\ea\ee
For $J_3$, we deduce from (\ref{kv712}) that
\be\la{kv715}\ba
& \left| \Lambda _{i, j}(\varphi (y), w(x)) u_i(y)\right| \\
& \quad \le \frac{C \left( |u_r(y)| + |u_z(y)| \right) }{ \vert \varphi (y)-w(x)\vert }
+ C\sum _{k=1}^2 \frac{ \vert (\partial _i w_k(x)-\partial _i \varphi _k(y)) u_i(y) \vert }{ \vert \varphi (y)-w(x)\vert ^2}.
\ea\ee
Moreover, for any $\varphi (x), \varphi (y)\in {\mathbb {D}}$ with $\varphi (x)\not =0$, we have
\be\la{kv716}\ba
\left| \varphi (y)-\frac{\varphi (x)}{\vert \varphi (x)\vert }\right| \leq \left| \varphi (y)-w(x)\right|,
\, \, \left| \varphi (y)-\varphi (x)\right| \leq \big \vert \varphi (y)- w(x) \big \vert,
\ea\ee
which together with (\ref{qyfj3}) gives
\be\la{kv717}\ba
\frac{C \left( |u_r(y)| + |u_z(y)| \right) }{ \vert \varphi (y)-w(x)\vert }
\le \frac{C \left( |u_r(y)| + |u_z(y)| \right) }{ |x-y| }.
\ea\ee

To estimate the second terms on the right-hand side of (\ref{kv715}),
we partition the boundary of $D_1$ into two components:
\be\ba\nonumber
\Gamma_1 & \triangleq \{ (r,z) \in \p D_1: r=1 \text{ or } r=2,-1<z<2 \},
\quad \Gamma_2 \triangleq \p D_1 \setminus  \Gamma_1.
\ea\ee

For any $x \in D$ and $y \in \Gamma_2$, we have $|x-y| \ge 1$.
It then follows from the arguments in Lemma \ref{l7} that $|\varphi(x) - \varphi(y)| \ge c_1$ for some constant $c_1 \in (0,1)$ depending only on $D_1$.

We then proceed to estimate the second terms on the right-hand side of (\ref{kv715}) by considering two distinct cases.

$Case \ 1: |\varphi(x)| \le 1-c_1.$
By the definition of $w(x)$, we derive
\be\la{kv718}\ba
\frac{ \vert \partial _i w_k(x)-\partial _i \varphi _k(y) \vert }{ \vert \varphi (y)-w(x)\vert ^2}
\le C \left| |\varphi(x)| \varphi(y) - \frac{\varphi(x)}{ |\varphi(x)| }\right|^{-2}.
\ea\ee
Note that
\be\la{kv719}\ba
\left| |\varphi(x)| \varphi(y) - \frac{\varphi(x)}{ |\varphi(x)| }\right|
\ge 1-|\varphi(x)| |\varphi(y)| \ge 1-|\varphi(x)| \ge c_1.
\ea\ee
Combining (\ref{yzgj1}), (\ref{kv715}), (\ref{kv717}), (\ref{kv718}), (\ref{kv719}), and Poincar\'e's inequality, we obtain
\be\la{kv720}\ba
|J_3| \le C \int_{D_1} \frac{\n ( u_r^2 + u_z^2 )}{|x-y|} dy
\le C \| \na \mathbf{u} \|^2_{L^2}.
\ea\ee

$Case \ 2: |\varphi(x)| > 1-c_1.$
First, we have
\be\la{kv721}\ba
\partial_{i}w_k(x)-\partial_{i}\varphi_k(y)
& =\frac{\partial_{i} \varphi_k(x)}{\vert \varphi (x)\vert ^2}
- \frac{2\varphi_k(x)\varphi _l(x)\partial_{i}\varphi_l(x)}{\vert \varphi (x)\vert ^4} -\partial_{i}\varphi_k(y).
\ea\ee
On the one hand, it follows from (\ref{qyfj3}) and (\ref{kv716}) that
\be\la{kv725}\ba
\left| \frac{\partial_{i}\varphi _k(x)}{\vert \varphi (x)\vert ^2} - \partial _{i}\varphi_k(y) \right|
& \le \left| \frac{\partial_{i}\varphi _k(x)}{\vert \varphi (x)\vert ^2} - \partial_{i}\varphi _k(x) \right|
+ |\partial_{i}\varphi_k(x) - \partial_{i}\varphi_k(y)| \\
& \le \left| \frac{1- \vert \varphi (x)\vert^2}{\vert \varphi (x)\vert^2} \partial_{i}\varphi _k(x) \right| + C |x-y| \\
& \le C \left( 1 - \vert \varphi (x) \vert \right)
+ C |\varphi(x)-\varphi(y)| \\
& \le C \left| \varphi (y)-w(x) \right|,
\ea\ee
where in the last inequality we have used the following fact:
\be\ba\nonumber
2\left| \varphi (y)-w(x) \right| \ge 1-\vert \varphi (x)\vert,
\ea\ee
due to (\ref{kv716}).

On the other hand, noticing that $|\varphi(x)| > 1-c_1$ gives
\be\la{kv726}\ba
\left| -\frac{2\varphi_k(x)\varphi_l(x)\partial_{i}\varphi_l(x) u_i(y)}{\vert \varphi(x)\vert^4} \right|
& \le C \left| \frac{\varphi_l(x)}{\vert \varphi (x)\vert} \partial_{i}\varphi_l(x) u_i(y) \right| \\
& = C \left| \varphi_l(x') \partial_{i}\varphi_l(x) u_i(y) \right|,
\ea\ee
where
\be\la{kv727}\ba
x' \triangleq \varphi^{-1} \left( \frac{\varphi(x)}{\vert \varphi (x)\vert} \right).
\ea\ee

Clearly, $x' \in \p D_1$.
Next, we show that in fact $x' \in \Gamma_1$.
From the selection of $c_1$, we conclude that for any $y \in \Gamma_2$, it holds that $|\varphi(x) - \varphi(y)| \ge c_1$.
We claim that $\frac{\varphi(x)}{|\varphi(x)|} \notin \varphi(\Gamma_2)$.
Otherwise, there exists $z \in \Gamma_2$ such that $\frac{\varphi(x)}{|\varphi(x)|} = \varphi(z)$, which implies
$|\varphi(x) - \varphi(z)| \ge c_1$.

However, we deduce from $|\varphi(x)| > 1-c_1$ that
\be\ba\nonumber
|\varphi(x) - \varphi(z)|
= \left|\varphi(x) - \frac{\varphi(x)}{|\varphi(x)|} \right| =1-|\varphi(x)| < c_1.
\ea\ee
This yields a contradiction, hence $x' \in \Gamma_1$.
Then the boundary condition $\mathbf{u} \cdot n=0$ on $\p \OM$ implies that $u_r(x')=0$.

Furthermore, since $(0,1)$ is the tangent vector at $x'$,
by \cite[Remark 2.1]{FLL} we conclude that $\p_2 \varphi(x')$ corresponds to the tangent vector at $\varphi(x')$,
which shows that $\p_2 \varphi_l(x') \varphi_l(x')=0$.
Thus, we have
\be\ba\nonumber
\varphi_l(x') \partial_{i}\varphi_l(x') u_i(x')=0,
\ea\ee
which yields
\be\la{kv731}\ba
& \left| \varphi_l(x') \partial_{i}\varphi _l(x) u_i(y) \right| \\
& = \left| \varphi_l(x') \partial_{i}\varphi _l(x) u_i(y)
- \varphi_l(x') \partial_{i}\varphi _l(x') u_i(x') \right| \\
& \le \left| \varphi_l(x') u_i(y) \right|
\left| \partial_{i}\varphi_l(x) - \partial_{i}\varphi_l(x') \right|
+ \left|\varphi_l(x') \partial_{i}\varphi_l(x') \right|
\left| u_i(y) - u_i(x') \right| \\
& \le C |x-x'| \left( |u_r(y)| + |u_z(y)| \right)
+ C \left| u_r(y) - u_r(x') \right| + C \left| u_z(y) - u_z(x') \right| \\
& \le C \left( |x-y| + |y-x'| \right) \left( |u_r(y)| + |u_z(y)| \right)
+ C \left| u_r(y) - u_r(x') \right| + C \left| u_z(y) - u_z(x') \right|.
\ea\ee
In addition, it follows from (\ref{qyfj3}), (\ref{kv716}), and (\ref{kv727}) that
\be\la{kv732}\ba
|y-x'| \le C |\varphi(y) - \varphi(x')|
= C \left|\varphi(y) - \frac{\varphi(x)}{|\varphi(x)|} \right|
\le C |\varphi(y) - w(x)|,
\ea\ee
and
\be\la{kv733}\ba
|y-x| \le C \left| \varphi (y)-\varphi (x)\right|
\le \big \vert \varphi (y)- w(x) \big \vert.
\ea\ee
By virtue of (\ref{kv721}), (\ref{kv725}), (\ref{kv726}), (\ref{kv731}),
(\ref{kv732}), and (\ref{kv733}), we obtain
\be\ba\nonumber
& \sum _{k=1}^2 \frac{ \vert (\partial_i w_k(x)
- \partial_i \varphi _k(y)) u_i(y) \vert }{ \vert \varphi (y)-w(x)\vert ^2} \\
& \le C \frac{ \left( |u_r(y)| + |u_z(y)| \right) }{ |x-y| }
+ C \frac{ \left| u_r(y) - u_r(x') \right| + \left| u_z(y) - u_z(x') \right| }{|x'-y|^2},
\ea\ee
which together with (\ref{kv715}) and (\ref{kv717}) gives
\be\ba\nonumber
|J_3| \le C \| \na \mathbf{u} \|^2_{L^2}
+ C \sum^{2}_{i,j=1} \int_{D_1} \frac{\left| u_i(x') - u_i(y) \right|}{|x'-y|^2} \n |u_j| (y) dy.
\ea\ee
Combining this with (\ref{kv79}), (\ref{kv711}), (\ref{kv714}), and (\ref{kv720}) leads to
\be\ba\nonumber
|J| \le C \| \na \mathbf{u} \|^2_{L^2}
+ C \sup_{x \in \ol{D_1}} \left( \sum^{2}_{i,j=1} \int_{D_1} \frac{\left| u_i(x) - u_i(y) \right|}{|x-y|^2} \n |u_j| (y) dy \right),
\ea\ee
which yields (\ref{kvl701}) and finishes the proof of Lemma \ref{ll7}.
\end{proof}

\begin{lemma}\la{l8}
For any $\ep>0$ and $0\le t_1<t_2$,
there exists a positive constant $C$ depending only on
$\ga$, $\beta$, $\mu$, $\ep$, $\|\n_0\|_{L^\infty}$, $\|\mathbf{u}_0\|_{H^1}$, and $K$,
such that when $\ga<2\beta$, it holds that
\be\ba\la{kv08}
\int_{t_1}^{t_2} - G(\mathbf{x}(t),t) dt
\le C R^{1+\ep}_T (t_2-t_1) + C R^{1+\frac{\beta}{4}+3\ep}_T
+ C R_T^{\frac{2+\beta}{3}},
\ea\ee
when $\ga \ge 2\beta$, we have
\be\ba\la{kv008}
\int_{t_1}^{t_2} - G(\mathbf{x}(t),t) dt
\le C R^{1+\frac{\beta}{4}+2\ep}_T (t_2-t_1+1) + C R_T^{\frac{2+\beta}{3}},
\ea\ee
where $\mathbf{x}(t)$ is the flow line determined by $\mathbf{x}(t)'=\mathbf{u}(\mathbf{x}(t),t)$.
\end{lemma}
\begin{proof}
First, we conclude from (\ref{kv701}) that
\be\la{kv81}\ba
- G(\mathbf{x}(t),t)
& \le \frac{d}{dt} \psi(t) + C \left( \| \sqrt{\n} \dot{\mathbf{u}} \|_{L^2}
+ \| \na \mathbf{u} \|^2_{L^2} + \| G \|_{H^1}
+ \| \na \mathbf{u} \|_{L^4} \right) + |J|.
\ea\ee
By virtue of (\ref{kv702}), (\ref{kv04}) and H\"older's inequality, we have
\be\ba\nonumber
| \psi(t) | & \le C \int_{D_1} |x-y|^{-1} \n (|u_r| + |u_z|) dy \\
& \le C \left( \int_{D_1} |x-y|^{-\frac{2+\nu}{1+\nu}} dy \right)^{\frac{1+\nu}{2+\nu}}
\left(\int_{D_1}\n^{2+\nu} (|u_r| + |u_z|)^{2+\nu} dy \right)^{\frac{1}{2+\nu}} \\
& \le C \nu^{-\frac{1+\nu}{2+\nu}} R_T^{\frac{1+\nu}{2+\nu}}
\left(\int_{\OM}\n |\mathbf{u}|^{2+\nu} d\mathbf{y} \right)^{\frac{1}{2+\nu}} \\
& \le C \nu^{-\frac{1+\nu}{2+\nu}} R_T^{\frac{1+\nu}{2+\nu}} \\
& \le C R_T^{\frac{2+\beta}{3}},
\ea\ee
which gives
\be\la{kv82}\ba
\int_{t_1}^{t_2} \frac{d}{dt} \psi(t) dt \le C R_T^{\frac{2+\beta}{3}}.
\ea\ee
Moreover, using (\ref{kv01}), (\ref{kv05}), (\ref{kv58}) and (\ref{kv06}), we derive
\be\la{kv83}\ba
& \int_{t_1}^{t_2} \left( \| \sqrt{\n} \dot{\mathbf{u}} \|_{L^2}
+ \| \na \mathbf{u} \|^2_{L^2} + \| G \|_{H^1}
+ \| \na \mathbf{u} \|_{L^4} \right) dt \\
& \le C \int_{t_1}^{t_2} \left( A_2
+ A^2_1 + R^{\frac{1}{2}}_T A_2 + R^{\frac{1}{4}+\ep}_T (1+A_1+A_2) \right) dt \\
& \le C \int_{t_1}^{t_2} \left( R_T ( 1+A^2_1 ) + \frac{A^2_2}{e+A^2_1} \right) dt \\
& \le C R^{1+\ep}_T (t_2-t_1+1).
\ea\ee
To estimate $|J|$, we recall from (\ref{kvl701}) that
\be\la{kv84}\ba
|J| & \le C \| \na \mathbf{u} \|^2_{L^2}
+ C \sup_{x \in \ol{D_1}} \left( \sum^{2}_{i,j=1} \int_{D_1} \frac{\left| u_i(x) - u_i(y) \right|}{|x-y|^2} \n |u_j| (y) dy \right).
\ea\ee
For any $x,y \in \overline{D_1}$,
the Sobolev embedding theorem (Theorem 4 of \cite[Chapter 5]{EL}) shows that for any $2<p<\infty$
\be\ba\nonumber
& |u_r(x) - u_r(y)| + |u_z(x) - u_z(y)| \\
& \le C(p) \left( \| \tilde{\na} u_r \|_{L^p(D_1)} + \| \tilde{\na} u_z \|_{L^p(D_1)} \right) |x-y|^{1-\frac{2}{p}} \\
& \le C(p) \| \na \mathbf{u} \|_{L^p(\OM)} |x-y|^{1-\frac{2}{p}},
\ea\ee
which implies
\be\la{kv88}\ba
& \sum^{2}_{i,j=1} \int_{D_1} \frac{\left| u_i(x) - u_i(y) \right|}{|x-y|^2} \n |u_j| (y) dy \\
& \le C \| \na \mathbf{u} \|_{L^p}
\int_{D_1} |x-y|^{-(1+\frac{2}{p})} \n ( |u_r| +|u_z|) dy.
\ea\ee

For $\de>0$ and $0<2s<1-\frac{2}{p}$, which will be determined later,
by applying H\"older's inequality and Lemma \ref{ewgn}, we derive
\be\la{kv89}\ba
& \int _{\vert x-y\vert<2\delta }\vert x-y\vert ^{-\left( 1+\frac{2}{p}\right) }\rho \vert u_r \vert (y)\,dy \\
& \leq C R_T \left( \int _{\vert x-y\vert <2\delta } \vert x-y \vert ^{-\left( 1+\frac{2}{p}\right) \frac{1}{1-s} } \,dy\right)
^{1-s} \Vert u_r \Vert _{L^{1/s}} \\
& \leq C R_T \delta ^{1-\frac{2}{p}-2s}
\left( s^{-\frac{1}{2}} \Vert u_r \Vert _{H^1} \right) \\
& \leq C s^{-\frac{1}{2}} R_T \left( A_1 \delta ^{1-\frac{2}{p}-2s} \right).
\ea\ee
In addition, we deduce from (\ref{kv04}) and H\"older's inequality that
\be\la{kv810}\ba
& \int _{\vert x-y\vert>\delta }\vert x-y\vert ^{-\left( 1+\frac{2}{p}\right) }\rho \vert u_r \vert (y)\,dy \\
& \leq C \left( \int _{\vert x-y\vert >\delta } \vert x-y\vert ^{-\left( 1+\frac{2}{p}\right) (\frac{2+\nu }{1+\nu })} \,dy\right)
^{\frac{1+\nu }{2+\nu }} \left( \int _\Omega \rho ^{2+\nu }\vert \mathbf{u} \vert ^{2+\nu }\,d \mathbf{x} \right) ^{\frac{1}{2+\nu }} \\
& \leq C R_T \delta ^{-\frac{2}{p}+\frac{\nu }{2+\nu }}.
\ea\ee
Now choose $\de>0$ such that
\be\la{kv811}\ba
\delta ^{-\frac{2}{p}+\frac{\nu }{2+\nu }} = A_1^{ \frac{2}{p} },
\ea\ee
For $\nu$ given by (\ref{kv004}) and any $2<p<6$, we set $2s=\frac{\nu}{2+\nu} \cdot \frac{p-2}{2}$,
which satisfies that $0<2s<1-\frac{2}{p}$.
Combining this with (\ref{kv811}) leads to
\be\la{kv812}\ba
A_1 \delta ^{1-\frac{2}{p}-2s} = A_1^{\frac{2}{p}}.
\ea\ee
Thus, from (\ref{kv89}), (\ref{kv810}), (\ref{kv811}) and (\ref{kv812}),
we conclude that for any $2<p<6$
\be\la{kv813}\ba
& \int_{D_1} |x-y|^{-(1+\frac{2}{p})} \n |u_r| dy \\
& \le \left( \int _{\vert x-y\vert<2\delta }
+ \int _{\vert x-y\vert>\delta } \right)
\vert x-y\vert ^{-\left( 1+\frac{2}{p}\right) }\rho \vert u_r \vert (y)\,dy \\
& \le C s^{-\frac{1}{2}} R_T A_1^{\frac{2}{p}} + C R_T A_1^{\frac{2}{p}} \\
& \le C R^{1+\frac{\beta}{4}}_T A_1^{\frac{2}{p}},
\ea\ee
where in the last inequality we have used
$s^{-\frac{1}{2}} \le C(p) \nu^{-1/2} \le C(p) R_T^{\frac{\beta}{4}}$ by (\ref{kv004}).

Similarly, we also have
\be\ba\nonumber
\int_{D_1} |x-y|^{-(1+\frac{2}{p})} \n |u_z|(y) dy
\le C R^{1+\frac{\beta}{4}}_T A_1^{\frac{2}{p}},
\ea\ee
which together with (\ref{kv88}) and (\ref{kv813}) shows that for any $2<p<6$
\be\la{kv815}\ba
\sum^{2}_{i,j=1} \int_{D_1} \frac{\left| u_i(x) - u_i(y) \right|}{|x-y|^2} \n |u_j| (y) dy
\le C \| \na \mathbf{u} \|_{L^p} R^{1+\frac{\beta}{4}}_T A_1^{\frac{2}{p}}.
\ea\ee

Next, by employing Lemma \ref{l5}, we estimate (\ref{kv815}) through the following two distinct cases:

$Case \ 1: \ga<2\beta.$
For any $\ep \in (0,\frac{1}{2})$, take $2<p<6$ sufficiently close to $2$ such that
\be\ba\nonumber
\frac{p}{2}< \min \left\{ \frac{\ga+1}{\ga},\ 
\frac{1+\beta/4+3\ep}{1+\beta/4+2\ep},\ \frac{1}{1-2\ep} \right\}.
\ea\ee
From (\ref{kv005}), we deduce that
\be\ba\nonumber
\| \nabla \mathbf{u} \|_{L^{p}}
& \le C R^{\frac{1}{2}-\frac{1}{p}+\ep}_T A_1^{\frac{2}{p}} (1+A_1+A_2)^{1-\frac{2}{p}}
\le C R^{2\ep}_T A_1^{\frac{2}{p}} (1+A_1+A_2)^{1-\frac{2}{p}},
\ea\ee
which together with (\ref{kv815}) and Young's inequality yields
\be\la{kv818}\ba
& \sum^{2}_{i,j=1} \int_{D_1} \frac{\left| u_i(x) - u_i(y) \right|}{|x-y|^2} \n |u_j| (y) dy \\
& \le C R^{1+\frac{\beta}{4}+2\ep}_T A_1^{\frac{4}{p}} (1+A_1+A_2)^{1-\frac{2}{p}} \\
& \le C R^{ (1+\frac{\beta}{4}+2\ep) \frac{p}{2} }_T A^2_1
+ C (1+A_1+A_2) \\
& \le C\left( 1+R^{1+\frac{\beta}{4}+3\ep}_T A^2_1+\frac{A^2_2}{e+A^2_1} \right).
\ea\ee
Integrating (\ref{kv818}) over $(t_1,t_2)$ and using (\ref{kv01}), (\ref{yzgj2}) and (\ref{kv06}) gives
\be\la{kv819}\ba
\int_{t_1}^{t_2} \sup_{x \in \ol{D_1}} \int_{D_1} \frac{\left| u_i(x) - u_i(y) \right|}{|x-y|^2} \n |u_j| (y) dy dt
\le C(t_2-t_1) + C R^{1+\frac{\beta}{4}+3\ep}_T.
\ea\ee

$Case \ 2: \ga \ge 2\beta.$
By virtue of (\ref{kv05}), we have
\be\ba\nonumber
\| \nabla \mathbf{u} \|_{L^{p}}
& \le C R^{\frac{1}{2}-\frac{1}{p}+\ep}_T (1+A_1)^{\frac{2}{p}}
(1+A_1+A_2)^{1-\frac{2}{p}},
\ea\ee
which along with (\ref{kv815}) and Young's inequality leads to
\be\la{kv821}\ba
& \sum^{2}_{i,j=1} \int_{D_1} \frac{\left| u_i(x) - u_i(y) \right|}{|x-y|^2} \n |u_j| (y) dy \\
& \le C R^{ \frac{3}{2}-\frac{1}{p}+\frac{\beta}{4}+\ep }_T
( A_1^{\frac{2}{p}} + A_1^{\frac{4}{p}} )
(1+A_1+A_2)^{1-\frac{2}{p}} \\
& \le C R^{ (\frac{3}{2}-\frac{1}{p}+\frac{\beta}{4}+\ep) \frac{p}{2} }_T (1+A^2_1)
+ C (1+A_1+A_2) \\
& \le C\left( R^{1+\frac{\beta}{4}+2\ep}_T (1+A^2_1) + \frac{A^2_2}{e+A^2_1} \right),
\ea\ee
provided $2<p<6$ sufficiently close to $2$ such that
$\frac{p}{2} \le \frac{3/2+\beta/4+2\ep}{3/2+\beta/4+\ep}$.

Integrating (\ref{kv821}) over $(t_1,t_2)$ and using (\ref{kv01}), (\ref{yzgj2}) and (\ref{kv06}), we arrive at
\be\la{kv822}\ba
\int_{t_1}^{t_2} \sup_{x \in \ol{D_1}} \int_{D_1} \frac{\left| u_i(x) - u_i(y) \right|}{|x-y|^2} \n |u_j| (y) dy dt
\le C R^{1+\frac{\beta}{4}+2\ep}_T (t_2-t_1+1).
\ea\ee
Finally, combining (\ref{kv81}), (\ref{kv82}), (\ref{kv83}), (\ref{kv84}), (\ref{kv819}), and (\ref{kv822}),
we obtain (\ref{kv08}) and (\ref{kv008}).
This completes the proof of Lemma \ref{l8}.
\end{proof}

\begin{lemma}\la{l9}
There exists a positive constant $C$ depending only on 
$\ga$, $\beta$, $\mu$, $\|\n_0\|_{L^\infty}$, $\|\mathbf{u}_0\|_{H^1}$, and $K$,
such that
\be\ba\la{kv09}
\sup_{0\leq t\leq T} \left( \|\n\|_{L^\infty} + \| \mathbf{u} \|_{H^1} \right)
+ \int_0^T \left( \| \mathbf{u} \|^2_{H^1} + \| \sqrt{\n} \dot{\mathbf{u}} \|^2_{L^2} \right) dt
\le C.
\ea\ee
\end{lemma}
\begin{proof}
First, we rewrite $(\ref{ns})_1$ using $(\ref{gw})$ as:
\be\la{kv91}\ba
\frac{d}{dt} \theta(\n) + P = -G + P(\ol{\n}),
\ea\ee
where $\theta(\n) = 2\mu \log\n + \frac{1}{\beta} \n^\beta$.

Since the function $y=\theta(\n)$ is strictly increasing on $(0,\infty)$,
its inverse function $\n=\theta^{-1}(y)$ exists for $y\in(-\infty,\infty)$.
We now express (\ref{kv91}) as:
\be\ba\nonumber
y'(t) = g(y) + h'(t),
\ea\ee
with
\be\la{kv93}\ba
y=\theta(\n), \quad g(y)=-P( \theta^{-1}(y) ), \quad h=\int_0^t \left( P(\ol{\n}) - G \right) ds.
\ea\ee
Note that $g(\infty)=-\infty$.
Next, we estimate $h$ in two cases.

$Case \ 1: \ga<2\beta.$
It follows from (\ref{kv01}) and (\ref{kv08}) that
\be\ba\nonumber
h(t_2)-h(t_1) \le C \left( R^{1+\frac{\beta}{4}+3\ep}_T + R_T^{\frac{2+\beta}{3}} \right)
+ C R^{1+\ep}_T (t_2-t_1).
\ea\ee
Then, we choose $N_0$, $N_1$ and $\overline{\zeta}$ in Lemma \ref{zli} as follows:
\be\la{kv95}\ba
N_0 = C \left( R^{1+\frac{\beta}{4}+3\ep}_T + R_T^{\frac{2+\beta}{3}} \right),\quad
N_1 = C R^{1+\ep}_T,\quad
\overline{\zeta} = \theta \left( \left( C R^{1+\ep}_T \right)^{1/\ga} \right),
\ea\ee
which together with (\ref{kv93}) implies
\be\ba\nonumber
g(\zeta)=-( \theta^{-1}(\zeta) )^\ga \le - N_1 = - C R^{1+\ep}_T
\quad \text{ for all } \zeta \ge \overline{\zeta}.
\ea\ee

Moreover, since $R_T \ge 1$, we have $\overline{\zeta} \le C R^{ (1+\ep)\frac{\beta}{\ga} }_T$.
Combining this with (\ref{kv95}) and Lemma \ref{zli}, we obtain
\be\la{kv97}\ba
R^\beta_T \le C R_T^{ \max\{ 1+\frac{\beta}{4}+3\ep,\frac{2+\beta}{3},
(1+\ep)\frac{\beta}{\ga} \} }.
\ea\ee
By virtue of $\beta>4/3$ and $\ga>1$,
we set $0<\ep<\min\{ (3\beta-4)/12,\ga-1 \}$, which along with (\ref{kv97}) shows
\be\la{kv98}\ba
\sup_{0 \le t \le T} \| \n \|_{L^\infty} \le C.
\ea\ee

$Case \ 2: \ga \ge 2\beta.$
From (\ref{kv01}) and (\ref{kv008}), we have
\be\ba\nonumber
h(t_2)-h(t_1) \le C \left( R^{1+\frac{\beta}{4}+2\ep}_T + R_T^{\frac{2+\beta}{3}} \right)
+ C R^{1+\frac{\beta}{4}+2\ep}_T (t_2-t_1).
\ea\ee
Next, we select $N_0$, $N_1$ and $\overline{\zeta}$ in Lemma \ref{zli} as:
\be\ba\nonumber
N_0 = C \left( R^{1+\frac{\beta}{4}+2\ep}_T + R_T^{\frac{2+\beta}{3}} \right),\quad
N_1 = C R^{1+\frac{\beta}{4}+2\ep}_T,\quad
\overline{\zeta} = \theta \left( \left( C R^{1+\frac{\beta}{4}+2\ep}_T \right)^{1/\ga} \right).
\ea\ee
Similarly, applying Lemma \ref{zli} yields
\be\la{kv911}\ba
R^\beta_T \le C R_T^{ \max\{ 1+\frac{\beta}{4}+2\ep,\frac{2+\beta}{3},
(1+\frac{\beta}{4}+2\ep)\frac{\beta}{\ga} \} }.
\ea\ee
Note that $\ga \ge 2\beta$ implies that
$(1+\frac{\beta}{4}+2\ep)\frac{\beta}{\ga} \le 1+\frac{\beta}{4}+2\ep$,
hence we conclude from (\ref{kv911}) that
\be\la{kv912}\ba
R^\beta_T \le C R_T^{ \max\{ 1+\frac{\beta}{4}+2\ep,\frac{2+\beta}{3} \} }.
\ea\ee
In view of $\beta>4/3$,
we choose $0<\ep<(3\beta-4)/8$. Then (\ref{kv912}) gives
\be\la{kv913}\ba
\sup_{0 \le t \le T} \| \n \|_{L^\infty} \le C.
\ea\ee
The combination of (\ref{kv98}), (\ref{kv913}), (\ref{kv01}), (\ref{kv06}), and 
Poincar\'e's inequality implies (\ref{kv09}) and finishes the proof of Lemma \ref{l9}.
\end{proof}

\section{A Priori Estimates (\uppercase\expandafter{\romannumeral2}): Higher Order Estimates}

This section is devoted to establishing some necessary higher-order estimates
for the axisymmetric strong solution of (\ref{ns})--(\ref{i3}) that satisfies (\ref{lct1}).
These estimates ensure that the strong solution can be extended globally in time.
The arguments are primarily adapted from \cite{CL,H1,FLL,HL2} with some modifications.

\begin{lemma}\la{gg1}
There exists a positive constant $C$ depending only on $\mu$, $\ga$, $\beta$, 
$\| \n_0 \|_{L^\infty}$, $\| \mathbf{u}_0 \|_{H^1}$, and $K$ such that
\be\ba\la{gkv01}
\sup_{0\le t\le T}
\si \int\n| \dot{\mathbf{u}} |^2dx
+\int_0^{T} \si \| \na \dot{\mathbf{u}} \|^2_{L^2} dt \le C,
\ea\ee
with $\si \triangleq \min\{ 1,t \}$.
Moreover, for any $p\in [1,\infty)$, there is a positive constant $C$ depending only on $p$, $\mu$, $\ga$, $\beta$, 
$\| \n_0 \|_{L^\infty}$, $\| \mathbf{u}_0 \|_{H^1}$, and $K$ such that
\be\ba\la{gkv02}
\sup_{1\le t\le T} \| \na \mathbf{u} \|_{L^p} \le C.
\ea\ee
\end{lemma}
\begin{proof}
The idea of this proof is adapted from \cite{CL,H1,FLL}.
Operating $ \dot{\mathbf{u}}^j[\frac{\pa}{\pa t}+\div(\mathbf{u}\cdot)]$ to
$(\ref{kv54})^j,$ summing with respect to $j,$
and integrating by parts over $\OM$, we obtain
\be\la{gkv11}\ba
\frac{d}{dt}\left(\frac{1}{2}\int\rho|\dot{\mathbf{u}}|^2dx \right)
&=\int \bigg( {\dot{\mathbf{u}}}\cdot \nabla G_t + {\dot{\mathbf{u}}}^j\mathrm {div}(\mathbf{u} \partial _jG) \bigg) dx\\
&\quad - \mu \int \bigg( {\dot{\mathbf{u}}} \cdot \na \times \curl \mathbf{u}_t
+ {\dot{\mathbf{u}}}^j\partial _k(\mathbf{u}^k ( \na \times \curl \mathbf{u} )^j ) \bigg) dx \\
&=I_1+I_2.
\ea\ee
For $I_1$, integration by parts combined with H\"older's and Young's inequalities yields
\be\la{gkv12}\ba
I_1 & = \int_{\partial \Omega} G_t ( \dot{\mathbf{u}} \cdot n) ds
- \int \div \dot{\mathbf{u}} \left( \dot{G} - \mathbf{u} \cdot \na G \right) dx
- \int \mathbf{u} \cdot \na \dot{\mathbf{u}}^j \p_j G dx \\
& \le \int_{\partial \Omega} G_t ( \dot{\mathbf{u}} \cdot n) ds
- \int \div \dot{\mathbf{u}} \dot{G} dx + C \| \na \dot{\mathbf{u}} \|_{L^2} \| \mathbf{u} \|_{L^6} \| \na G \|_{L^3} \\
& \le  \int_{\partial \Omega} G_t ( \dot{\mathbf{u}} \cdot n) ds
- \int \div \dot{\mathbf{u}} \dot{G} dx
+ \varepsilon \Vert \nabla \dot{\mathbf{u}} \Vert_{L^2}^2
+ C(\ep) ( A^2_1 + A^2_2 ),
\ea\ee
where in the last inequality we have used the following estimate:
\be\la{gkv14}\ba
& \| \na G \|_{L^3}+\| \na \curl \mathbf{u} \|_{L^3} \\
& \le \| \na G \|^{\frac{1}{2}}_{L^2} \| \na G \|^{\frac{1}{2}}_{L^6}
+\| \na \curl \mathbf{u} \|^{\frac{1}{2}}_{L^2} \| \na \curl \mathbf{u} \|^{\frac{1}{2}}_{L^6} \\
& \le C \left( A_1 + A_2 \right)^{\frac{1}{2}} 
\left(\| \n \dot{\mathbf{u}}\|_{L^6} + \| \na \mathbf{u} \|_{L^6} \right)^{\frac{1}{2}} \\
& \le C \left( 1 + A_2 \right)^{\frac{1}{2}}
\left( \| \sqrt{\n} \dot{\mathbf{u}}\|_{L^2}
+ \| \na \dot{\mathbf{u}}\|_{L^2} +1+ A_1 + A_2 \right)^{\frac{1}{2}} \\
& \le C \left( 1 + A_2 \right)^{\frac{1}{2}}
\left( 1+A_2 + \| \na \dot{\mathbf{u}}\|_{L^2} \right)^{\frac{1}{2}} \\
& \le C \left( 1 + A_2 + \| \na \dot{\mathbf{u}}\|^{\frac{1}{2}}_{L^2}
+A^{\frac{1}{2}}_2 \| \na \dot{\mathbf{u}}\|^{\frac{1}{2}}_{L^2}\right),
\ea\ee
due to (\ref{pt1}), (\ref{kv05}), (\ref{kv56}), (\ref{kv57}), (\ref{kv58}), and (\ref{kv09}).

Next, for the boundary term in (\ref{gkv12}), using (\ref{i3}) and (\ref{bjds}), we derive
\be\la{gkv15}\ba
&\int _{\partial \Omega }G_t({\dot{\mathbf{u}}}\cdot n) ds\\ 
&= - \int _{\partial \Omega }G_t( \mathbf{u} \cdot \na n \cdot \mathbf{u}) ds \\
&= -\frac{d}{dt} \int _{\partial \Omega } G ( \mathbf{u} \cdot \na n \cdot \mathbf{u} )ds
+\int _{\partial \Omega } G( \mathbf{u} \cdot \na n \cdot \mathbf{u} )_t ds\\ 
&= -\frac{d}{dt} \int _{\partial \Omega } G( \mathbf{u} \cdot \na n \cdot \mathbf{u} )ds
+ \int _{\partial \Omega } G({\dot{\mathbf{u}}} \cdot \na n \cdot \mathbf{u} )
+ G( \mathbf{u} \cdot \nabla n \cdot {\dot{\mathbf{u}}}) ds  \\
&\quad - \int _{\partial \Omega } G \left( ( \mathbf{u} \cdot \nabla \mathbf{u} )\cdot \nabla n \cdot \mathbf{u} \right) ds
- \int _{\partial \Omega} G \left( \mathbf{u} \cdot \nabla n \cdot ( \mathbf{u} \cdot \nabla \mathbf{u} ) \right) ds \\
&=-\frac{d}{dt} \int_{\partial \Omega } G ( \mathbf{u} \cdot \nabla n \cdot \mathbf{u} )ds + J_1+J_2+J_3.
\ea\ee
It follows from (\ref{pt1}), (\ref{kv58}), (\ref{kv09}), and Poincar\'e's inequality that
\be\la{gkv16}\ba
J_1 & = \int _{\partial \Omega } G({\dot{\mathbf{u}}} \cdot \na n \cdot \mathbf{u} )
+ G( \mathbf{u} \cdot \nabla n \cdot {\dot{\mathbf{u}}}) ds \\
& \leq C \Vert G \Vert_{H^1} \Vert \dot{\mathbf{u}} \Vert_{H^1} \Vert \mathbf{u} \Vert_{H^1} \\
& \leq C ( A_1 + A_2 )
\left( \Vert \sqrt{\n} \dot{\mathbf{u}} \Vert _{L^2} + \Vert \nabla \dot{\mathbf{u}} \Vert _{L^2} \right) \\
&\leq \varepsilon \Vert \nabla \dot{\mathbf{u}} \Vert _{L^2}^2
+ C(\ep) ( A^2_1 + A^2_2 ).
\ea\ee
By virtue of (\ref{bjds}), (\ref{kv613}), (\ref{kv58}),
(\ref{kv09}) and H\"older's inequality, we arrive at
\be\la{gkv17}\ba
\vert J_2\vert= & \left| - \int _{\partial \Omega } G \left( ( \mathbf{u} \cdot \nabla \mathbf{u} )\cdot \nabla n \cdot \mathbf{u} \right) ds \right| \\
= & \left| \int _{\partial \Omega} \mathbf{u}^\bot \times n \cdot \nabla \mathbf{u}^i \partial_i n_j \mathbf{u}^j G ds \right| \\
= & \left| \int _{\partial \Omega} n \cdot ( \na \mathbf{u}^i \times \mathbf{u}^\bot ) \partial_i n_j \mathbf{u}^j G ds \right| \\
= & \left| \int \div \left( ( \na \mathbf{u}^i \times \mathbf{u}^\bot ) \partial_i n_j \mathbf{u}^j G \right) dx \right| \\
= & \left| \int  \na (\partial_i n_j \mathbf{u}^j G ) \cdot ( \na \mathbf{u}^i \times \mathbf{u}^\bot )
- ( \na \mathbf{u}^i \cdot \na \times \mathbf{u}^\bot ) \partial_i n_j \mathbf{u}^j G dx \right| \\
\le & C \int \vert \nabla \mathbf{u} \vert \left( \vert G \vert \vert \mathbf{u} \vert^2
+ \vert G \vert \vert \mathbf{u} \vert \vert \nabla \mathbf{u} \vert
+ \vert \mathbf{u} \vert^2 \vert \nabla G \vert \right) dx \\
\le & C\Vert \nabla \mathbf{u} \Vert _{L^4} \left( \Vert G \Vert_{L^4} \Vert \mathbf{u} \Vert _{L^4}^2
+ \Vert G \Vert _{L^4} \Vert \mathbf{u} \Vert_{L^4} \Vert \nabla \mathbf{u} \Vert _{L^4}
+ \Vert \nabla G \Vert _{L^2} \Vert \mathbf{u} \Vert_{L^8}^2\right) \\
\le & C \left(\Vert \nabla \mathbf{u} \Vert _{L^4} \Vert \nabla \mathbf{u} \Vert^2_{L^2}
+ \Vert \nabla \mathbf{u} \Vert^2_{L^4} \Vert \nabla \mathbf{u} \Vert_{L^2} \right) \Vert G \Vert _{H^1} \\
\le & C \Vert \nabla \mathbf{u} \Vert^2_{L^4} \Vert \nabla \mathbf{u} \Vert_{L^2}
\left( A_1 + A_2 \right) \\
\le & C A^2_1 + C A^2_2 + C \Vert \nabla \mathbf{u} \Vert^4_{L^4}.
\ea\ee
Similarly, we also have
\be\ba\nonumber
\vert J_3 \vert \le C A^2_1 + C A^2_2 + C \Vert \nabla \mathbf{u} \Vert^4_{L^4}.
\ea\ee
Combining this with (\ref{gkv15}), (\ref{gkv16}) and (\ref{gkv17}) leads to
\be\la{gkv19}\ba
\int _{\partial \Omega} G_t ({\dot{\mathbf{u}}}\cdot n) ds 
\leq & -\frac{d}{dt} \int_{\partial \Omega } G ( \mathbf{u} \cdot \nabla n \cdot \mathbf{u} )ds
+ \ep \Vert \nabla {\dot{\mathbf{u}}}\Vert _{L^2}^2
+ C(\ep) ( A^2_1 + A^2_2 ),
\ea\ee
where we have used the following estimate:
\be\la{gkv180}\ba
\| \na \mathbf{u} \|^4_{L^4}
& \le C \left( \| \div \mathbf{u} \|^4_{L^4} + \| \curl \mathbf{u} \|^4_{L^4} + \| \mathbf{u} \|^4_{L^4} \right) \\
& \le C \left( \| G \|^4_{L^4} + \| P-P(\ol{\n}) \|^4_{L^4}
+ \| \curl \mathbf{u} \|^4_{L^4} + \| \na \mathbf{u} \|^4_{L^2} \right) \\
& \le C \left( \| G \|^2_{L^2} \| G \|^2_{H^1} + \| P-P(\ol{\n}) \|^2_{L^2}
+ \| \curl \mathbf{u} \|^2_{L^2} \| \curl \mathbf{u} \|^2_{H^1} + \| \na \mathbf{u} \|^2_{L^2} \right) \\
& \le C \left( \| G \|^2_{H^1} + \| \curl \mathbf{u} \|^2_{H^1} + A^2_1 \right) \\
& \le C \left( A^2_1 + A^2_2 \right),
\ea\ee
owing to (\ref{ewgn01}), (\ref{ewgn02}), (\ref{dc1}), (\ref{kv58}), and (\ref{kv09}).

For the second term on the last line of (\ref{gkv12}), from $(\ref{ns})_1$ and (\ref{gw}), we deduce that
\be\ba\nonumber
\dot{G} & = G_t + \mathbf{u} \cdot \na G \\
& = \lam_t \div \mathbf{u} + (2\mu+\lam) \div \mathbf{u}_t
+ \mathbf{u} \cdot \na ( (2\mu+\lam) \div \mathbf{u} ) - P_t - \mathbf{u} \cdot \na P \\
& = (\lam_t + \mathbf{u} \cdot \na \lam ) \div \mathbf{u}
+ (2\mu+\lam) \div \dot{\mathbf{u}}
- (2\mu+\lam) \div ( \mathbf{u} \cdot \na \mathbf{u} ) \\
& \quad + (2\mu+\lam) \mathbf{u} \cdot \na \div \mathbf{u} + \ga P \div \mathbf{u}
 \\
& = -\n \lam'(\n)(\div \mathbf{u})^2 + (2\mu+\lam) \div \dot{\mathbf{u}}
- (2\mu+\lam) \p_i \mathbf{u}^j \p_j \mathbf{u}^i
+ \ga P \div \mathbf{u},
\ea\ee
which together with Young's inequality yields
\be\la{gkv109}\ba
- \int \div \dot{\mathbf{u}} \dot{G} dx
& = - \int (2\mu+\lam) (\div \dot{\mathbf{u}})^2 dx
+ \int \n \lam'(\n)(\div \mathbf{u})^2 \div \dot{\mathbf{u}} dx \\
& \quad + \int (2\mu+\lam) \p_i \mathbf{u}^j \p_j \mathbf{u}^i \div \dot{\mathbf{u}} dx
-\ga \int P \div \mathbf{u} \div \dot{\mathbf{u}} dx \\
& \le -2\mu \| \div \dot{\mathbf{u}} \|^2_{L^2} + \ep \| \na \dot{\mathbf{u}} \|^2_{L^2}
+ C(\ep) \| \na \mathbf{u} \|^2_{L^2} + C(\ep) \| \na \mathbf{u} \|^4_{L^4}.
\ea\ee
This combined with (\ref{gkv12}), (\ref{gkv19}), and (\ref{gkv180}) gives
\be\la{gkv110}\ba
I_1 & \le -\frac{d}{dt} \int_{\partial \Omega } G ( \mathbf{u} \cdot \nabla n \cdot \mathbf{u} ) ds
-2\mu \| \div \dot{\mathbf{u}} \|^2_{L^2} + 3\ep \| \na \dot{\mathbf{u}} \|^2_{L^2} + C(\ep) ( A^2_1 + A^2_2 ).
\ea\ee
For $I_2$, integrating by parts and using (\ref{kv09}), (\ref{gkv14}) and (\ref{gkv180}), we arrive at
\be\la{gkv111}\ba
I_2 & = - \mu \int \bigg( {\dot{\mathbf{u}}} \cdot \na \times \curl \mathbf{u}_t
+ {\dot{\mathbf{u}}}^j\partial _k(\mathbf{u}^k ( \na \times \curl \mathbf{u} )^j ) \bigg) dx \\
& = \mu \int_{\p \OM} \curl \mathbf{u}_t \times n \cdot \dot{\mathbf{u}} ds
- \mu \int \curl \dot{\mathbf{u}} \cdot \curl \mathbf{u}_t dx
+ \mu \int \mathbf{u} \cdot \na \dot{\mathbf{u}} \cdot (\na \times \curl \mathbf{u}) dx \\
& = \mu \int_{\p \OM} \curl \mathbf{u}_t \times n \cdot \dot{\mathbf{u}} ds
- \mu \int |\curl \dot{\mathbf{u}}|^2 dx + \mu \int \mathbf{u} \cdot \na \curl \mathbf{u} \cdot \curl \dot{\mathbf{u}} dx \\
& \quad + \mu \int \curl \dot{\mathbf{u}} \cdot (\na \mathbf{u}^i \times \p_i \mathbf{u}) dx
+ \mu \int \mathbf{u} \cdot \na \dot{\mathbf{u}} \cdot (\na \times \curl \mathbf{u}) dx \\
& \le \mu \int_{\p \OM} \curl \mathbf{u}_t \times n \cdot \dot{\mathbf{u}} ds
- \mu \int |\curl \dot{\mathbf{u}}|^2 dx + C \| \na \dot{\mathbf{u}} \|_{L^2} \| \na \mathbf{u} \|^2_{L^4} \\
& \quad + C \| \na \dot{\mathbf{u}} \|_{L^2} \| \na \curl \mathbf{u} \|_{L^3} \| \mathbf{u} \|_{L^6} \\
& \le \mu \int_{\p \OM} \curl \mathbf{u}_t \times n \cdot \dot{\mathbf{u}} ds
- \mu \| \curl \dot{\mathbf{u}} \|^2_{L^2} + \ep \Vert \nabla \dot{\mathbf{u}} \Vert_{L^2}^2
+ C(\ep) ( A^2_1 + A^2_2 ),
\ea\ee
where in the third equality we have used the following fact:
\be\ba\nonumber
\curl (\mathbf{u} \cdot \na \mathbf{u}) = \mathbf{u} \cdot \na \curl \mathbf{u} + \na \mathbf{u}^i \times \p_i \mathbf{u}.
\ea\ee

Next, we deal with the boundary term of (\ref{gkv111}).
In view of (\ref{i3}), (\ref{pt1}), (\ref{kv613}), (\ref{gkv180}) and Young's inequality, we derive
\be\ba\nonumber
\mu \int_{\p \OM} \curl \mathbf{u}_t \times n \cdot \dot{\mathbf{u}} ds
& = - \mu \int_{\p \OM} \mathbf{u}_t \cdot K \cdot \dot{\mathbf{u}} ds \\
& = - \mu \int_{\p \OM} \dot{\mathbf{u}} \cdot K \cdot \dot{\mathbf{u}} ds
+ \mu \int_{\p \OM} (\mathbf{u} \cdot \na \mathbf{u}) \cdot K \cdot \dot{\mathbf{u}} ds \\
& = - \mu \int_{\p \OM} \dot{\mathbf{u}} \cdot K \cdot \dot{\mathbf{u}} ds
+ \mu \int_{\p \OM} \mathbf{u}^\bot \times n \cdot \nabla \mathbf{u}^i (K^i \cdot \dot{\mathbf{u}}) ds \\
& = - \mu \int_{\p \OM} \dot{\mathbf{u}} \cdot K \cdot \dot{\mathbf{u}} ds
+ \mu \int_{\p \OM} n \cdot ( \na \mathbf{u}^i \times \mathbf{u}^\bot ) (K^i \cdot \dot{\mathbf{u}}) ds \\
& = -\mu \int_{\p \OM} \dot{\mathbf{u}} \cdot K \cdot \dot{\mathbf{u}} ds
+ \mu \int \div ( ( \na \mathbf{u}^i \times \mathbf{u}^\bot ) (K^i \cdot \dot{\mathbf{u}}) )dx \\
& = - \mu \int_{\p \OM} \dot{\mathbf{u}} \cdot K \cdot \dot{\mathbf{u}} ds
- \mu \int ( \na \mathbf{u}^i \cdot \na \times \mathbf{u}^\bot ) (K^i \cdot \dot{\mathbf{u}}) dx \\
& \quad + \mu \int  \na (K^i \cdot \dot{\mathbf{u}}) \cdot ( \na \mathbf{u}^i \times \mathbf{u}^\bot ) dx \\
& \le - \mu \int_{\p \OM} \dot{\mathbf{u}} \cdot K \cdot \dot{\mathbf{u}} ds
+ C \| \dot{\mathbf{u}} \|_{L^2} \| \na \mathbf{u} \|^2_{L^4}
+ C \| \na \dot{\mathbf{u}} \|_{L^2} \| \na \mathbf{u} \|^2_{L^4} \\
& \le - \mu \int_{\p \OM} \dot{\mathbf{u}} \cdot K \cdot \dot{\mathbf{u}} ds
+ \ep \| \na \dot{\mathbf{u}} \|^2_{L^2} + C(\ep) ( A^2_1 + A^2_2 ).
\ea\ee
Combining this with (\ref{gkv111}) yields
\be\la{gkv114}\ba
I_2 & \le - \mu \int_{\p \OM} \dot{\mathbf{u}} \cdot K \cdot \dot{\mathbf{u}} ds
- \mu \| \curl \dot{\mathbf{u}} \|^2_{L^2}
+ 2\ep \Vert \nabla \dot{\mathbf{u}} \Vert_{L^2}^2 + C(\ep) ( A^2_1 + A^2_2 ).
\ea\ee
From (\ref{gkv11}), (\ref{gkv110}) and (\ref{gkv114}), we conclude that
\be\la{gkv116}\ba
& \frac{1}{2} \frac{d}{dt} \left( \int \rho|\dot{\mathbf{u}}|^2 dx \right)
+ 2\mu \| \div \dot{\mathbf{u}} \|^2_{L^2}
+ \mu \| \curl \dot{\mathbf{u}}\|^2_{L^2}
+ \mu \int_{\p \OM} \dot{\mathbf{u}} \cdot K \cdot \dot{\mathbf{u}} ds \\
& \le -\frac{d}{dt} \int_{\partial \Omega } G ( \mathbf{u} \cdot \nabla n \cdot \mathbf{u} ) ds
+ 5 \varepsilon \Vert \nabla {\dot{\mathbf{u}}}\Vert _{L^2}^2
+ C(\ep) ( A^2_1 + A^2_2 ).
\ea\ee
In addition, the boundary conditions (\ref{i3}) show that
\be\ba\nonumber
\left( \dot{\mathbf{u}} + (\mathbf{u} \cdot \na n) \times \mathbf{u}^\bot \right) \cdot n = 0 \quad \text{ on } \p \OM.
\ea\ee
Then, we define $\mathbf{v} \triangleq \dot{\mathbf{u}} + (\mathbf{u} \cdot \na n) \times \mathbf{u}^\bot$, which implies that $\mathbf{v} \cdot n = 0$ on $\p \OM$.
By Lemma \ref{dltdc2}, we obtain
\be\la{gkv118}\ba
2 \mu \| D(\mathbf{v}) \|^2_{L^2} = 2 \mu \| \div \mathbf{v} \|^2_{L^2}
+ \mu \| \curl \mathbf{v} \|^2_{L^2}
- 2 \mu \int_{\p \OM} \mathbf{v} \cdot D(n) \cdot \mathbf{v} ds.
\ea\ee
Moreover, noticing that Young's inequality gives
\be\ba\nonumber
& 2 \mu \| \div \mathbf{v} \|^2_{L^2} + \mu \| \curl \mathbf{v} \|^2_{L^2} \\
& \le 2 \mu \| \div \dot{\mathbf{u}} \|^2_{L^2} + \mu \| \curl \dot{\mathbf{u}} \|^2_{L^2}
+ C \| \na \dot{\mathbf{u}} \|_{L^2} \| \na \mathbf{u} \|^2_{L^4}
+ C \| \na \mathbf{u} \|^4_{L^4} \\
& \le 2 \mu \| \div \dot{\mathbf{u}} \|^2_{L^2} + \mu \| \curl \dot{\mathbf{u}} \|^2_{L^2}
+ \ep \| \na \dot{\mathbf{u}} \|^2_{L^2} + C(\ep) \| \na \mathbf{u} \|^4_{L^4},
\ea\ee
which along with (\ref{gkv118}) implies
\be\la{gkv121}\ba
& 2 \mu \| D(\mathbf{v}) \|^2_{L^2} + 2 \mu \int_{\p \OM} \mathbf{v} \cdot D(n) \cdot \mathbf{v} ds \\
& \le 2 \mu \| \div \dot{\mathbf{u}} \|^2_{L^2} + \mu \| \curl \dot{\mathbf{u}}\|^2_{L^2}
+ \ep \| \na \dot{\mathbf{u}} \|^2_{L^2} + C(\ep) \| \na \mathbf{u} \|^4_{L^4}.
\ea\ee
On the other hand, Young's inequality and (\ref{gkv180}) ensure that
\be\la{gkv122}\ba
\mu \int_{\p \OM} \mathbf{v} \cdot K \cdot \mathbf{v} ds
\le \mu \int_{\p \OM} \dot{\mathbf{u}} \cdot K \cdot \dot{\mathbf{u}} ds
+\ep \| \na \dot{\mathbf{u}} \|^2_{L^2} + C(\ep) ( A^2_1 + A^2_2 ).
\ea\ee
Combining (\ref{gkv180}), (\ref{gkv116}), (\ref{gkv121}), and (\ref{gkv122}), we have
\be\la{gkv123}\ba
& \frac{1}{2} \frac{d}{dt} \left( \int \rho|\dot{\mathbf{u}}|^2 dx \right)
+ 2 \mu \| D(\mathbf{v}) \|^2_{L^2}
+ \mu \int_{\p \OM} \mathbf{v} \cdot ( K + 2D(n) ) \cdot \mathbf{v} ds \\
& \le -\frac{d}{dt} \int_{\partial \Omega } G ( \mathbf{u} \cdot \nabla n \cdot \mathbf{u} ) ds
+ 7 \varepsilon \Vert \nabla {\dot{\mathbf{u}}}\Vert _{L^2}^2
+ C(\ep) ( A^2_1 + A^2_2 ).
\ea\ee
Furthermore, from the definition of $\mathbf{v}$ and Lemma \ref{dltdc1}, we derive
\be\la{gkv124}\ba
\| \na \dot{\mathbf{u}} \|^2_{L^2}
& \le C \| \na \mathbf{v} \|^2_{L^2} + C \| \na \mathbf{u} \|^4_{L^4} \\
& \le C \left( 2 \| D(\mathbf{v}) \|^2_{L^2}
+ \int_{\p \OM} \mathbf{v} \cdot ( K+2D(n) ) \cdot \mathbf{v} ds \right)
+ C ( A^2_1 + A^2_2 ),
\ea\ee
which together with (\ref{gkv123}) yields
\be\ba\nonumber
& \frac{1}{2} \frac{d}{dt} \left( \int \rho|\dot{\mathbf{u}}|^2 dx \right)
+ 2 \mu \| D(\mathbf{v}) \|^2_{L^2}
+ \mu \int_{\p \OM} \mathbf{v} \cdot ( K + 2D(n) ) \cdot \mathbf{v} ds \\
& \le -\frac{d}{dt} \int_{\partial \Omega } G ( \mathbf{u} \cdot \nabla n \cdot \mathbf{u} ) ds
+ C \ep \left( 2 \mu \| D(\mathbf{v}) \|^2_{L^2}
+ \mu \int_{\p \OM} \mathbf{v} \cdot ( K+2D(n) ) \cdot \mathbf{v} ds \right) \\
& \quad + C(\ep) ( A^2_1 + A^2_2 ).
\ea\ee
Therefore, taking $\ep$ suitably small and multiplying $\si$, we arrive at
\be\la{gkv126}\ba
& \frac{1}{2} \frac{d}{dt} \left( \si \int \rho|\dot{\mathbf{u}}|^2 dx \right)
+ \mu \si \| D(\mathbf{v}) \|^2_{L^2}
+ \frac{\mu}{2} \si \int_{\p \OM} \mathbf{v} \cdot ( K + 2D(n) ) \cdot \mathbf{v} ds \\
& \le -\frac{d}{dt} \left( \si \int_{\partial \Omega } G ( \mathbf{u} \cdot \nabla n \cdot \mathbf{u} ) ds \right)
+ C ( A^2_1 + A^2_2 ),
\ea\ee
where we have used the following estimate:
\be\la{gkv127}\ba
\left| \int_{\partial \Omega} G (\mathbf{u} \cdot \nabla n \cdot \mathbf{u}) ds \right|
& \le C \| G\|_{H^1} \| \na \mathbf{u} \|^2_{L^2} \\
& \le C \left( \| \sqrt{\n} \dot{\mathbf{u}} \|_{L^2} + A_1 \right)
\| \na \mathbf{u} \|_{L^2} \\
& \le \frac{1}{4} \| \sqrt{\n} \dot{\mathbf{u}} \|^2_{L^2} + C A^2_1,
\ea\ee
due to (\ref{kv58}), (\ref{kv09}) and Young's inequality.

Integrating (\ref{gkv126}) over $(0,T)$ and using (\ref{kv09}), (\ref{yzgj2}) and (\ref{gkv127}) gives
\be\la{gkv130}\ba
& \sup_{0 \le t \le T} \si \int \rho|\dot{\mathbf{u}}|^2dx 
+ \int_0^T \si \left( \| D(\mathbf{v}) \|^2_{L^2} + \int_{\p \OM} \mathbf{v} \cdot ( K + 2D(n) ) \cdot \mathbf{v} ds \right) dt
\le C.
\ea\ee
In addition, from (\ref{gkv124}), (\ref{gkv130}) and (\ref{kv09}), we deduce that
\be\ba\nonumber
& \int_0^T \si \| \na \dot{\mathbf{u}} \|^2_{L^2} dt \\
& \le C \int_0^T \si \left( \| D(\mathbf{v}) \|^2_{L^2}
+ \int_{\p \OM} \mathbf{v} \cdot ( K + 2D(n) ) \cdot \mathbf{v} ds \right) dt
+ C \int_0^T ( A^2_1 + A^2_2 ) dt \\
& \le C,
\ea\ee
which together with (\ref{gkv130}) yields (\ref{gkv01}).

Finally, (\ref{kv05}) and (\ref{kv09}) ensure that for any $1<p<\infty$,
\be\la{gkv132}\ba
\| \na \mathbf{u} \|_{L^p} \le C + C \| \sqrt{\n} \dot{\mathbf{u}} \|_{L^2},
\ea\ee
which together with (\ref{gkv130}) implies (\ref{gkv02}) and completes the proof of Lemma \ref{gg1}.
\end{proof}

Next, using the uniform estimates (\ref{kv09}), (\ref{gkv02}), and Lemma \ref{ewgn}, we can derive the following exponential decay,
whose proof is similar to that of \cite[Proposition 4.2]{FLW}.
\begin{lemma}\la{kve}
For any $p \in [1,\infty)$, there exist positive constants
$C$ and $\alpha_0$ depending only on
$p$, $\ga$, $\beta$, $\mu$, $\| \n_0 \|_{L^\infty}$, $\| \mathbf{u}_0 \|_{H^1}$, and $K$, such that for any $1 \le t <\infty$,
\be\la{kve1}\ba
\| \n(\cdot,t)-\ol{\n_0}\|_{L^p} +
\| \na \mathbf{u}(\cdot,t) \|_{L^p} \le C e^{-\alpha_0 t}.
\ea\ee
\end{lemma}

\begin{lemma}\la{gkv2}
There exists a positive constant $C$ depending only on 
$T$, $q$, $\ga$, $\beta$, $\mu$,
$\| \mathbf{u}_0 \|_{H^1}$, $\| \rho_0 \|_{W^{1,q}}$, and $K$, such that
\be\la{gkv201}\ba
&\sup_{0\le t\le T} \left( \| \n \|_{W^{1,q}} + t \| u \|^2_{H^2} \right) \\
& + \int_0^T \left( \|\nabla^2 u\|^{(q+1)/q}_{L^q}
+ t \|\nabla^2 u\|_{L^q}^2+t\| u_t\|_{H^1}^2 \right) dt\le C.
\ea\ee
\end{lemma}
\begin{proof}
First, we define $\Phi = (\Phi^1,\Phi^2,\Phi^3)$
with $\Phi^i \triangleq (2\mu+\lam(\n)) \p_i \n$ $(i=1,2,3)$.
By virtue of $(\ref{ns})_1$, we find that $\Phi^{i}$ satisfies
\be\la{gkv21}\ba
\p_t \Phi^i + (\mathbf{u} \cdot \na) \Phi^i
+ (2\mu+\lam(\n)) \na \n \cdot \p_i \mathbf{u} + \n \p_i G + \n \p_i P
+ \Phi^i \div \mathbf{u} = 0.
\ea\ee
Then, multiplying (\ref{gkv21}) by $|\Phi|^{q-2} \Phi^i$ and
integrating by parts over $\OM$, we derive by (\ref{i3})
\be\la{gkv22}\ba
\frac{d}{dt} \| \Phi \|_{L^q}
& \le C ( 1 + \| \na \mathbf{u} \|_{L^\infty} ) \| \Phi \|_{L^q}
+ C \| \na G \|_{L^q}.
\ea\ee
In addition, it follows from (\ref{kv56}), (\ref{kv57}), (\ref{gkv132}), and Sobolev embedding that
\be\la{gkv23}\ba
& \| \div \mathbf{u} \|_{L^\infty}+\| \curl \mathbf{u} \|_{L^\infty} \\
& \le C \left( \| G \|_{L^\infty} + \| P-P(\ol{\n}) \|_{L^\infty} \right)
+ \| \curl \mathbf{u} \|_{L^\infty} \\
& \le C + C \left( \| G \|_{L^2} + \| \na G \|_{L^q}
+ \| \curl \mathbf{u} \|_{L^2}+\| \na \curl \mathbf{u} \|_{L^q} 
+ \| \na \mathbf{u} \|_{L^q} \right) \\
& \le C \left( 1 + \| \n \dot{\mathbf{u}} \|_{L^q} \right).
\ea\ee
By (\ref{dc1}), (\ref{gkv132}), (\ref{gkv23}), (\ref{kv09}), (\ref{kv56}), and (\ref{kv57}), we obtain that for any $p\in[2,q]$,
\be\la{gkv25}\ba
\|\na^2 \mathbf{u}\|_{L^p}
& \le C \left( \| \div \mathbf{u} \|_{W^{1,p}}
+ \| \curl \mathbf{u} \|_{W^{1,p}} + \| \mathbf{u} \|_{L^p} \right) \\
& \le C \left( \| \na \mathbf{u} \|_{L^p}
+ \| \na \div \mathbf{u} \|_{L^p}
+ \| \na \curl \mathbf{u} \|_{L^p} \right) \\
& \le C + C \left( \| \na ( (2\mu+\lam) \div \mathbf{u} ) \|_{L^p}
+ \| \div \mathbf{u} \|_{L^{\frac{pq}{q-p}}} \| \na \n \|_{L^q}
+ \| \n \dot{\mathbf{u}} \|_{L^p} \right) \\
& \le C\left( 1+\| \div \mathbf{u} \|_{L^{\frac{pq}{q-p}}} \right) \| \na \n \|_{L^q}
+ C \left( \| \na G \|_{L^p}+\| \n \dot{\mathbf{u}} \|_{L^p} \right) \\
& \le C \left( 1+\| \n \dot{\mathbf{u}} \|_{L^q} \right) \| \na \n \|_{L^q}
+ C \| \n \dot{\mathbf{u}} \|_{L^p}.
\ea\ee
Combining this with (\ref{gkv23}), (\ref{kv09}) and Lemma \ref{bkm} yields
\be\ba\la{gkv24}
\|\na \mathbf{u}\|_{L^\infty} 
& \le C \left( \|\div \mathbf{u} \|_{L^\infty}
+ \|\curl \mathbf{u}\|_{L^\infty} \right) \log \left(e+ \|\na^2 \mathbf{u}\|_{L^q} \right)+ C\|\na \mathbf{u}\|_{L^2}+C \\
& \le C \left( 1 + \| \n \dot{\mathbf{u}} \|_{L^q} \right)
\log \left(e + \| \na \n \|_{L^q} + \| \n \dot{\mathbf{u}} \|_{L^q}
+ \| \n \dot{\mathbf{u}} \|_{L^q} \| \na \n \|_{L^q} \right) \\
& \le C \left( 1 + \| \n \dot{\mathbf{u}} \|_{L^q} \right)
\log \left(e + \| \na \n \|_{L^q} \right)
+ C \| \n \dot{\mathbf{u}} \|^{1+1/q}_{L^q}.
\ea\ee
Moreover, with the definition of $\Phi$ and (\ref{kv09}), we have
\be\la{gkv26}\ba
2\mu \| \na \n \|_{L^q} \le \| \Phi \|_{L^q} \le C \| \na \n \|_{L^q},
\ea\ee
which together with (\ref{gkv22}) and (\ref{gkv24}) implies
\be\la{gkv27}\ba
\frac{d}{dt} \log( e + \| \Phi \|_{L^q} )
& \le C \left( 1 + \| \n \dot{\mathbf{u}} \|_{L^q} \right)
\log \left(e + \| \Phi \|_{L^q} \right)
+ C \| \n \dot{\mathbf{u}} \|^{1+1/q}_{L^q}.
\ea\ee
Meanwhile, we deduce from (\ref{ewgn01}), (\ref{pt1}) and H\"{o}lder's inequality that
\be\ba\nonumber
\| \rho \dot{\mathbf{u}} \|_{L^q} 
& \le C\| \rho \dot{\mathbf{u}} \|_{L^2}^{2(q-1)/(q^2-2)}
\| \dot{\mathbf{u}} \|_{L^{q^2}}^{q(q-2)/(q^2-2)} \\
& \le C\| \rho \dot{\mathbf{u}} \|_{L^2}^{2(q-1)/(q^2-2)}
\| \dot{\mathbf{u}} \|_{H^1}^{q(q-2)/(q^2-2)} \\
& \le C\| \rho^{1/2} \dot{\mathbf{u}} \|_{L^2}
+ C\| \rho \dot{\mathbf{u}} \|_{L^2}^{2(q-1)/(q^2-2)}
\| \na \dot{\mathbf{u}} \|_{L^2}^{q(q-2)/(q^2-2)},
\ea\ee
which together with (\ref{gkv01}) and (\ref{pt1}) gives
\be\la{gkv29}\ba
&\int_0^T \left(\| \rho \dot{\mathbf{u}} \|^{1+1 /q}_{L^q}
+ t \| \dot{\mathbf{u}} \|^2_{H^1} \right) dt \\
&\le C+C \int_0^T\left( \| \rho^{1/2} \dot{\mathbf{u}} \|_{L^2}^2
+ t\| \na \dot{\mathbf{u}} \|_{L^2}^2
+ t^{-(q^3-q^2-2p)/(q^3-q^2-2p+2)} \right) dt \\ 
&\le C.
\ea\ee
Applying Gr\"onwall's inequality to (\ref{gkv27}) and using (\ref{gkv26}), (\ref{gkv29}), we arrive at
\be\la{gkv210}\ba
\sup_{0 \le t \le T} \| \n \|_{W^{1,q}} \le C,
\ea\ee
which together with (\ref{gkv01}), (\ref{gkv132}), (\ref{gkv25}) and (\ref{gkv29}) leads to
\be\la{gkv211}\ba
\sup_{0\le t\le T} t \| \na^2 \mathbf{u} \|^2_{L^2}
+ \int_0^T \left( \| \nabla^2 \mathbf{u} \|^{(q+1)/q}_{L^q}
+ t \|\nabla^2 \mathbf{u} \|_{L^q}^2 \right) dt
\le C.
\ea\ee
Finally, we apply (\ref{pt1}), (\ref{kv09}), (\ref{gkv01}),
(\ref{gkv180}), (\ref{gkv211}), and H\"older's inequality to derive that
\be\ba\nonumber
\int_0^T t \| \mathbf{u}_t \|^2_{H^1} dt
& \le C \int_0^T t \left( \| \sqrt{\n} \mathbf{u}_t \|^2_{L^2}
+ \| \na \mathbf{u}_t \|^2_{L^2} \right) dt \\
& \le C \int_0^T t \left( \| \sqrt{\n} \dot{\mathbf{u}} \|^2_{L^2}
+ \| \mathbf{u} \cdot \na \mathbf{u} \|^2_{L^2}
+ \| \na \dot{\mathbf{u}} \|^2_{L^2}
+ \| \na ( \mathbf{u} \cdot \na \mathbf{u} )\|_{L^2}^2 \right) dt \\
& \le C + C \int_0^T t \left( \| \mathbf{u} \|_{L^4}^2 \| \na \mathbf{u} \|_{L^4}^2
+ \| \mathbf{u} \|_{L^{2q/(q-2)}}^2\|\nabla^2 \mathbf{u} \|_{L^q}^2
+ \| \nabla \mathbf{u} \|_{L^4}^4 \right) dt \\
& \le C,
\ea\ee
which together with (\ref{gkv210}) and (\ref{gkv211}) yields (\ref{gkv201})
and completes the proof of Lemma \ref{gkv2}.
\end{proof}

\section{Proofs of Theorems \ref{th0}--\ref{th3}}
With all the a priori estimates established in Sections 3 and 4, we now prove the main results of this paper.
In fact, the proofs of Theorems \ref{th0}--\ref{th3} are routine; we only sketch them here and refer to \cite{FLW,HL2,VK,CL} for complete details.

We first state the global existence of
strong solution to problem (\ref{ns})--(\ref{i3}) provided that (\ref{bg}) holds
and $(\n_0,\mathbf{m}_0)$ satisfies (\ref{lct01}).
The proof follows that of \cite[Proposition 5.1]{HL2} with minor modifications.
\begin{proposition}\la{51}
Assume that $(\ref{bg})$ holds and that the initial data $\left( \n _0,\mathbf{m}_0 \right)$ satisfy $(\ref{lct01})$.
Then the problem $(\ref{ns})-(\ref{i3})$ admits a unique
strong solution $(\n,\mathbf{u})$ within the axisymmetric class
in $\OM \times (0,\infty)$ satisfying $(\ref{lct1})$ and $(\ref{lct2})$ for any $0<T<\infty$.
Moreover, for $q>3$, $(\n,\mathbf{u})$ satisfies $(\ref{gkv201})$ with some positive constant $C$ depending only on
$T$, $q$, $\ga$, $\beta$, $\mu$, $\| \mathbf{u}_0 \|_{H^1}$, $\| \rho_0 \|_{W^{1,q}}$, and $K$.
\end{proposition}

Proof of Theorem \ref{th0}.
Let $(\n_0,\mathbf{m}_0)$ be the initial data in Theorem \ref{th0}, satisfying (\ref{ssol1}).
By standard approximation (see \cite{EL}), there exists a sequence of functions
$(\hat{\n}^{\de}_0,\hat{\mathbf{u}}^\de_0)\in C^\infty$
that are axisymmetric and periodic in $x_3$ with period 1, such that
\be\ba\nonumber
\lim_{\de \to 0} \left( \| \hat{\n}^{\de}_0 - \n_0 \|_{W^{1,q}}
+ \| \hat{\mathbf{u}}^\de_0 - \mathbf{u}_0 \|_{H^1} \right)=0.
\ea\ee
However, $\hat{\mathbf{u}}^\de_0$ may not satisfy the slip boundary conditions.
To address this, we define $\mathbf{u}^\de_0$ as the unique smooth solution to the following elliptic equation:
\be\la{pp1}\ba
\begin{cases}
\Delta \mathbf{u}^\de_0
= \Delta \hat{\mathbf{u}}^\de_0 &\, \text{ in } \OM, \\
\mathbf{u}^\de_0 \cdot n = 0, \, \curl \mathbf{u}^\de_0 \times n 
= -K \mathbf{u}^\de_0 & \, \text{ on } \p \OM.
\end{cases}
\ea\ee
Define $\n^\de_0=\hat{\n}^{\de}_0 + \de$ and $\mathbf{m}^\de_0 = \n^\de_0 \mathbf{u}^\de_0$.
The standard arguments (see \cite{L1}) yield
\be\ba\nonumber
\lim_{\de \to 0} \left( \| \n^{\de}_0 - \n_0 \|_{W^{1,q}}
+ \| \mathbf{u}^\de_0 - \mathbf{u}_0 \|_{H^1} \right)=0.
\ea\ee

By Proposition \ref{51}, the problem (\ref{ns})--(\ref{i3}), in which the initial data $(\n_0,\mathbf{m}_0)$ are replaced by $(\n_0^\de,\mathbf{m}^\de_0)$,
admits a unique global strong solution  $(\n^\de,\mathbf{u}^\de)$
satisfying (\ref{gkv201}) for any $0<T<\infty$
with some positive constant $C$ independent of $\de$.
Then, letting $\de\rightarrow 0$
and using standard compactness arguments (see \cite{HL2,LZZ,P,VK}), we obtain that the problem  (\ref{ns})--(\ref{i3})
has a global strong solution $(\n,\mathbf{u})$ satisfying (\ref{ssol2}).
Moreover, (\ref{kve1}) implies that $(\n,\mathbf{u})$ satisfies the estimate (\ref{ed}).
The uniqueness of the solution $(\n,\mathbf{u})$ satisfying (\ref{ssol2})
follows from arguments analogous to those in \cite{Ge}.
This completes the proof of Theorem \ref{th0}.

Using the compactness techniques developed in \cite{HL2,VK},
Theorem \ref{th1} can be proved in the same manner as Theorem \ref{th0},
and we omit the details.

Proof of Theorem \ref{th3}. 
The proof of Theorem \ref{th3} is similar to that of \cite[Theorem 1.2]{CL} and is also omitted.

\bigskip

\noindent\textbf{Data availability.} No data was used for the research described in the article.

\bigskip

\noindent\textbf{Conflict of interest.} The author declares no conflict of interest.

\end{document}